\documentclass[11pt,reqno]{amsart}

\usepackage[top=1in, bottom=1in, left=1in,right=1in]{geometry}

\usepackage[utf8]{inputenc}
\usepackage{amsfonts, amsmath, amssymb, mathrsfs}
\usepackage{amsthm, url}
\usepackage[all]{xy}
\usepackage{graphicx, 
epstopdf}
\usepackage{subfig}
\usepackage{color}
\usepackage{bbm}
\usepackage{subfloat}
\usepackage[draft]{fixme}
\usepackage{bm} 
\usepackage{epigraph}
\usepackage{endnotes}
\usepackage{ifpdf}
\usepackage{nccmath}

\usepackage{multicol}
\usepackage{multirow}
\usepackage{graphicx}
\usepackage{tikz}

\usepackage{csquotes}

\usepackage{hyperref}
 \usepackage[fit]{truncate}

\usepackage[T1]{fontenc}
\usepackage{cleveref}

\newtheorem{theorem}{Theorem}[section]
\newtheorem*{theorem*}{Theorem}

\newtheorem{lemma}[theorem]{Lemma}

\newtheorem{corollary}[theorem]{Corollary}
\newtheorem{proposition}[theorem]{Proposition}

\theoremstyle{definition}
\newtheorem{definition}[theorem]{Definition} 
\newtheorem{question}[theorem]{Question}
\newtheorem{exampleth}[theorem]{Example}
\newenvironment{example}{\begin{exampleth}}{\hfill
    $\diamond$\\ \end{exampleth}}
\newtheorem{remarkth}[theorem]{Remark}
\newenvironment{remark}{\begin{remarkth}}{\hfill $\diamond$\\ \end{remarkth}}

\DeclareMathOperator{\conv}{conv}

\DeclareMathOperator{\trop}{trop}
\DeclareMathOperator{\val}{val}

\DeclareMathOperator{\tcone}{tcone}
\DeclareMathOperator{\cone}{cone}
\DeclareMathOperator{\NP}{NewtonPolytope}
\DeclareMathOperator{\supp}{supp}

\DeclareMathOperator{\Mid}{Mid}
\DeclareMathOperator{\topint}{int}

\newcommand{\R}{\mathbb{R}}

\newcommand{\N}{\mathbb{N}}
\newcommand{\Z}{\mathbb{Z}}

\newcommand{\A}{\mathcal{A}}

\newcommand{\RR}{\mathbb R}

\newcommand{\ZZ}{\mathbb Z}
\newcommand{\NN}{\mathbb N}

\newcommand{\cR}{\mathcal R}

\newcommand{\wh}{\widehat}
\newcommand{\qm}{ {\rm QM}}
\newcommand{\po}{ {\rm PO}}

\usepackage{xparse}
\NewDocumentCommand{\M}{m m o}{%
  \IfNoValueTF{#3}
    {M^{#1}_{#2}} 
    {M^{#1}_{#2}(#3)} 
}

\begin{document}
\title{Moments, Sums of Squares, and Tropicalization}

\author{Grigoriy Blekherman} 
\address{School of Mathematics, Georgia Institute of Technology, Atlanta GA, USA}
\email {greg@math.gatech.edu}
\author{Felipe Rinc\'on}
\address{School of Mathematical Sciences, Queen Mary University of London, London, UK}
\email{f.rincon@qmul.ac.uk}
\author{Rainer Sinn}
\address{Mathematisches Institut, Universit\"at Leipzig, Leipzig, Germany}
\email{rainer.sinn@uni-leipzig.de}
 \author{Cynthia Vinzant}
 \address{Department of Mathematics, University of Washington, Seattle, WA, USA}
 \email{vinzant@uw.edu}
 \author{Josephine Yu}
 \address{School of Mathematics, Georgia Institute of Technology,
        Atlanta GA, USA}
      \email {jyu@math.gatech.edu}

       \maketitle

 \begin{center}
 Dedicated to Bernd Sturmfels on the occasion of his 60th birthday.
\end{center}

\begin{abstract}
We use tropicalization to study the duals to cones of nonnegative
polynomials and sums of squares on a semialgebraic set $S$.   The
truncated cones of moments of measures supported on the set $S$ is dual to
nonnegative polynomials on $S$, while ``pseudo-moments'' are dual to sums
of squares approximations to nonnegative polynomials. We provide explicit combinatorial descriptions of
tropicalizations of the moment and pseudo-moment cones, and
demonstrate their usefulness in distinguishing between nonnegative
polynomials and sums of squares.
We give examples that show new limitations of sums of squares approximations of nonnegative polynomials.
When the semialgebraic set is defined by binomial inequalites, its
moment and pseudo-moment cones are closed under Hadamard product. In this case, their
tropicalizations are polyhedral cones that encode all binomial
inequalities on the moment and pseudo-moment cones.



 \end{abstract}


\section{Introduction}
Understanding nonnegativity of polynomials in terms of sums of squares has been a central challenge in real algebraic geometry dating back to the work of Hilbert. The dual side of this problem is important in analysis and known as the moment problem. We now take a moment to introduce it.

For a semialgebraic set $S\subseteq \RR^n$ and a finite subset $A\subset \N^n$ considered as exponent vectors of monomials, we consider the convex cone $M_A(S)$ of $A$-moments of measures supported on $S$:
\[M_A(S) = \left\{ \left(\int_S x^a d\mu \right)_{a \in A} : \mu \text{ is a nonnegative Borel measure supported on } S\right\}.\]
The membership testing problem for $M_A(S)$ is known as the (truncated) moment problem. 
Despite extensive work this cone can be explicitly described in very few situations even when $S=\mathbb{R}^n$ and $A$ corresponds to all moments of degree at most $2d$ \cite{MR1303090, MR1147276,MR3782989, SchmudgenBook}.  

The convex dual cone of $M_A(S)$ consists of polynomials with support in $A$ which are nonnegative on $S$. An important tool for understanding $M_A(S)$ comes from \emph{Positivstellens\"atze} in real algebraic geometry: theorems on representing the nonnegative polynomials via sums of squares \cite{MR1092173,MR1254128}.
A subset of the nonnegative cone is the cone $\Sigma(S)_A$  of ``obviously nonnegative'' polynomials generated by sums of squares.
We will call elements in its dual $\Sigma(S)_A^\vee$ ``pseudo-moments''. Tropicalization of the moment cone and pseudo-moment cone gives us ``combinatorial shadows'' of these sets. Our explicit descriptions of these shadows lead to interesting combinatorial questions, some of which have been considered in the context of SONC polynomials \cite{Reznick89,IlimanDeWolff, katthan2021unified}.

Another way of understanding our results is through binomial inequalities in moments and pseudo-moments of measures supported on $S$. When the semialgebraic set $S$ is closed under Hadamard multiplication, the tropicalization $\trop M_A(S)$ of the moment cone is a rational polyhedral cone. Its dual cone $(\trop M_A(S))^\vee$ encodes all of the binomial inequalities in $A$-moments. Similarly, binomial moment inequalities that can proved via sums of squares correspond to another rational polyhedral cone, which may depend on a degree bound for the sums of squares construction. While polynomial inequalities valid on $M_A(S)$ are difficult to characterize, we can explicitly describe all binomial inequalities in moments and pseudo-moments by finding the extreme rays of the corresponding rational polyhedral cones. The use of tropicalizations to analyze the power of sums of squares method was first introduced in \cite{BRST} for analyzing graph density inequalities, and further developed in \cite{BR21}. We take inspiration from some of their results and techniques, for instance the use of the Hadamard property to ensure that the tropicalization is a convex cone. However, to the best of our knowledge, this is the first instance where tropicalization is used to study the relationship between the moment and pseudo-moment cones.  We believe that this just scratches the surface of applications of tropicalization in semi-algebraic geometry.

We start with a pair of examples which illustrate our setup and results. 

\begin{example}[Motzkin Configuration on the Nonnegative Orthant]
Let $S=\RR^2_{\geq 0}$ be the nonnegative orthant and let $A \subset \NN^2$ be the \emph{Motzkin configuration}: 
$A=\{(0,0), (1,2), (2,1), (1,1) \}$, which gives us the exponents of moments we are recording:
\[m_{00}=\int_S 1\, d\mu, \quad m_{12}=\int_S xy^2 \, d\mu, \quad m_{21}=\int_S x^2y \, d\mu, \quad m_{11}=\int_S xy\, d\mu.
\]
There is only one binomial inequality satisfied by $A$-moments of measures supported on $S$:
\begin{equation}\label{eq:logconv} m_{00}m_{12}m_{21}\geq m_{11}^3.\end{equation}
If we regard moments as functions on $A$, then we see that moments are nonnegative \emph{log-convex functions} on $A$, and in fact inequalities coming from log-convexity are the only possibly binomial inequalities in $A$-moments for measures supported on the nonnegative orthant $\mathbb{R}_{\geq 0}^n$ (see \Cref{thm:genmom}).\\

We now consider $A$-pseudo-moments of measures supported on $\RR^n_{\geq 0}$. Pseudo-moments are defined as linear functionals that are nonnegative on ``obviously'' nonnegative polynomials coming from sums of squares (see \Cref{sec:SOS}). We show in \Cref{tropicalizationpseudononneg} that $A$-pseudo-moments of measures supported on $\RR^n_{\geq 0}$ satisfy \emph{log-midpoint-convexity inequalities}:
\begin{equation} \label{eqn:midptlog}m_\alpha m_\beta \geq m_{\left(\frac{\alpha+\beta}{2}\right)}^2, \end{equation}
with $\alpha,\beta,\frac{\alpha+\beta}{2} \in A$. Moreover these inequalities generate all possible binomial inequalities valid on $A$-pseudo-moments. Since the Motzkin configuration contains no midpoints, we see that there are no binomial inequalities valid on $A$-pseudo-moments. 
\end{example}

\begin{remark} The combinatorial notions of convex and midpoint-convex functions on $A$ are quite similar to what has been developed for analyzing certain sparse globally nonnegative polynomials and sums of squares arising from the arithmetic mean-geometric mean inequality. Such polynomials were originally called AGI-forms by Reznick in \cite{Reznick89} and were later called Sum of Nonnegative Circuit Polynomials (SONC) in \cite{IlimanDeWolff}. The only difference is that for analyzing global nonnegativity, it makes a difference whether points in $A$ have all even coordinates or not, and for instance midpoints convexity has to hold only between even points in $A$. As we will see in \Cref{thm:whole} and \Cref{tropicalizingpseudomoments}, this is precisely what happens for us as well when analyzing measures supported on all of $\R^n$.
\end{remark}

\begin{example}[Motzkin Configuration on the Square.]\label{ex:motzsq}
Let $S=[0,1]^2\subset \RR^2$ be the unit square given by inequalities $0\leq x\leq 1$, $0\leq y \leq 1$. Let $A \subset \mathbb{N}^2$ again be the Motzkin configuration. 
In addition to the log-convexity inequality \eqref{eq:logconv}, the following binomial moment inequalities are naturally valid on the unit square, since all variables lie between $0$ and $1$:
\[ m_{00} \geq m_{11}, \quad m_{11}\geq m_{12}, \quad m_{11}\geq m_{21}.\]
Any binomial inequality in $A$-moments of measures supported on $S$ can be obtained from the above inequalities and \eqref{eq:logconv} via exponentiation and multiplication (see Example \ref{exm:motzcube}). 

As we increase the degree $d$, sums of squares provide increasingly better approximations to polynomials supported on $A$ that are nonnegative on $S$, and thus can, in principle, be used to provide increasingly sharper binomial inequalities for pseudo-moments (see Section \ref{sec:tsos} for more details). If we regard pseudo-moments as functions on $A$ then increasing the degree allows us to use moments that lie outside of $A$. For instance we can show that $m_{00}m_{12}\geq 2m_{11}$ by combining the inequality $m_{12} \geq m_{22}$ with the log-midpoint-convexity inequality \eqref{eqn:midptlog}: $m_{00}m_{22} \geq m_{11}^2$.

We show that, in this case, the binomial $A$-pseudo-moment inequalities stabilize, and only the following binomial inequalities can be learned via sums of squares (regardless of the degree $d$):
\[ m_{11}\geq m_{12}, \quad m_{11}\geq m_{21}, \quad m_{00}m_{12}\geq m_{11}^2, \quad m_{00}m_{21}\geq m_{11}^2.\]
Therefore, for any degree $d$, sums of squares cannot prove the moment inequality $m_{00}m_{12}m_{21}\geq m_{11}^3$, and moreover, sums of squares remain quantifiably far away from certifying this inequality. 
\end{example}

\begin{remark}
Since the unit square is compact, it follows from Schm\"{u}dgen's Positivstellensatz \cite{MR1092173} that any polynomial $f$ strictly positive on the unit square has a sum of squares certificate. Therefore, as the degree increases, sums of squares provide an increasingly better approximation to all nonnegative polynomials supported on $A$. However, as we have seen, tropicalizations stabilize, and higher degree sums of squares do not have larger tropicalizations. This is due to the fact that $\trop (S)$ only depends on the neighborhood of zero and the ``neighborhood of infinity'' contained in $S$. We give a simple example of this phenomenon below: Let $S$ be the planar triangle with vertices $(0,0)$, $(1,0)$ and $(1,1)$ and let $S_\varepsilon$ be the quadrilateral with vertices $(0,0)$, $(1,0)$, $(1,1)$ and $(0,\varepsilon)$. Then we have $S_\varepsilon \rightarrow S$ as $\varepsilon \rightarrow 0$, however $\trop (S_\varepsilon)$ is the nonpositive orthant for all $\varepsilon >0$, and $\trop(S)$ is the subset of the nonpositive orthant given by the inequality $x\geq y$. 
\end{remark}

\begin{remark}
The unit square is special in that all nonnegative polynomials have a sum of squares certificate. Example \ref{ex:motzsq} also shows that even though every nonnegative polynomial is a sum of squares, there does not exist a degree bound for the certificate even for just the $A$-supported polynomials \cite[9.4.6 Example (1)]{MR2383959}. 
\end{remark}

\subsection{Main Results in Detail:} 
Our main results are about the tropicalizations of moment cones and pseudo-moment cones for semi-algebraic sets with the Hadamard property. 
We say that a subset $S$ of $\RR^n_{\geq 0}$ has the Hadamard property if $S$ is closed under coordinatewise (Hadamard) multiplication. Concretely, we focus on nonnegative orthants, hypercubes, and toric cubes to discuss our general results. Throughout, we fix a finite set $A\subset \NN^n$ of exponents and consider the $A$-moments: $m_a = \int_S x^a$ for $a\in A$ (also known as truncated moment sequences). 

We think of elements of the tropicalization of the moment cone (resp.~pseudo-moment cone) as functions $h\colon A\to \R$ and describe the tropicalization mainly in terms of discrete convexity properties of these functions. For the moment cone, we have a general description of the tropicalization of $M_A(S)$ for any subset of the nonnegative orthant with the Hadamard property:
\begin{theorem*}[{\Cref{thm:genmom}}]
     Let $S\subset \R_{\geq 0}^n$ be a semi-algebraic subset with the Hadamard property such that $S\subset \overline{S\cap \R^n_{>0}}$. 
    The tropicalization of the $A$-moment cone $M_A(S)$ is the rational polyhedral cone of functions $h \colon A\to \R$ satisfying 
\[
\sum_{i=1}^r \lambda_i h(a_i) \geq h(b)\text{ for all }a_1,\ldots,a_r,b\in A, \lambda_i \geq 0, \sum_{i=1}^r \lambda_i = 1, \text{ with } \sum_{i=1}^r \lambda_ia_i - b \in \trop(S)^\vee.
\]	
\end{theorem*}

In particular, all of the functions $h:A\to \R$ in $M_A(S)$ satisfy the following: 
\begin{enumerate}
		\item (Convexity) $\sum_{i=1}^r \lambda_i h(a_i) \geq h(b)$ for all $a_i,b\in A$, $\lambda_i \geq 0$, $\sum_{i=1}^r \lambda_i = 1$, $\sum_{i=1}^r\lambda_i a_i =b$;
		\item  (Nonincreasing) $h(a)\geq h(b)$ whenever $a-b \in \trop(S)^\vee$.
	\end{enumerate}
The first type of inequality is the naive form of discrete convexity that arises in this context. The second type of inequality is where the set $S$ enters: 
The tropicalization of $S$ is a rational polyhedral cone and $\trop(S)^\vee$ is its dual cone, which defines a partial order of $\R^A$  -- and the second inequality says that the functions in the tropicalization are order preserving in this sense. 
In the case $S = \R_{\geq 0}^n$, we have $\trop(S)^\vee = \{0\}$, so the tropicalizations of the $A$-moment cones do not include inequalities of type (2). For $S = [0,1]^n$, we get $\trop(S)^\vee = \R_{\leq 0}^n$ and the inequalities of type (2) say that the functions $h\in \trop(M_A(S))$ are non-increasing in the coordinate directions.
However, the condition in \Cref{thm:genmom} is stronger than the combination of conditions (1) and (2), see \Cref{exm:nottwo}.

We can also think of this result as a general description of all binomial inequalities valid on the moment (by exponentiation). With the analogous result for pseudo-moment cones, we will see that these inequalities suffice in distinguishing moments from pseudo-moments in many important cases. Moreover, there is a rich combinatorial interplay between the geometry of the moment configuration $A$ and the geometric and algebraic description of $S$.

We study the convex cone of convex functions in the sense of the above theorem in \Cref{sec:convexity} from the point of view of discrete and tropical geometry. In particular, we show in \Cref{thm:K_{A,C}AsMax} that the convex cone of functions $h\colon A \to \R$ with $\sum \lambda_i h(a_i) \geq h(b)$ for all $a_i,b\in A$ satisfying $\sum \lambda_i a_i - b \in \trop(S)^\vee$ is the tropical conical hull of $A(\trop(S)^\vee) = \{Au\colon u\in \trop(S)^\vee\}$, where we think of $A$ as a matrix.

We now move on to pseudo-moment cones, which are the dual cones to truncated preorderings or quadratic modules. We describe in detail how we truncate (in a total degree version) at the beginning of \Cref{sec:tsos}. For pseudo-moment cones, we focus on the case that the semialgebraic set $S$ has an inequality description in terms of pure binomial inequalities.

\begin{theorem*}[{\Cref{thm:genpseudomom}}]
  Let $g_1, \hdots, g_r$ be pure binomials and consider the semi-algebraic set $S\subset\R^n$ defined by $x_i \geq 0$ and $g_j\geq 0$. 
  Assume $S$ is full dimensional, $S \subset \overline{S\cap \R_{>0}^n}$, and that the vectors $w_i = a_i - b_i$, where $g_i = x^{a_i}-x^{b_i}$, generate the semigroup $N = \trop(S)^\vee \cap \Z^n$. 
  Then, for any integer $d\geq 0$ the tropicalization of $\qm_d(g_1,\dots,g_r)^\vee$ is the rational polyhedral cone given by the following inequalities: 
  \begin{enumerate}
	  \item (Midpoint convexity:) $h(u_1) + h(u_2) \geq 2 h(v)$ for all $u_1,u_2,v$ such that $|u_i| \leq d$, $|v|\leq d$ and $u_1 + u_2 = 2v$;
	  \item (Nonincreasing:) $h(u) \geq h(v)$ whenever $|u|\leq d$, $|v|\leq d$, and $u-v\in \trop(S)^\vee$. 
  \end{enumerate}
  The inequalities in $A$-pseudo-moments provable by sums of squares of degree at most $d$ are dual to the coordinate projection of $F(S)_d$ onto the coordinates indexed by $A$.
\end{theorem*}
In the case of pseudo-moments, we need the additional assumption on the inequality description of $S$ that the exponent vectors of the inequalities generate the semigroup of lattice points in the convex cone $\trop(S)^\vee$ to give the same inequalities of type (2) as in the case of moment cones. This is an assumption that, from a purely theoretical point of view, can be made without loss of generality by adding valid and redundant inequalities, if necessary. Without this assumption, we only get some inequalities of type (2), namely those corresponding to the lattice points in $\trop(S)^\vee$ that also lie in the semigroup generated by the exponent vectors.

\Cref{sec:convexity} contains a comparison of the polyhedral convex cones given by convexity (the tropicalization of the moment cone) and given by mid-point convexity (the tropicalization of the pseudo-moment cone).

Our most intriguing observation is that tropicalizations of pseudomoment cones stabilize as the degree bound $d$ grows. This means that for sufficiently large $d$ the tropicalizations of pseudomoment cones remain the same, even though pseudomoments themselves provide a convergent approximation to the moment cone. This phenomenon was already observed in \Cref{ex:motzsq}.
We provide an explicit description of when stabilization occurs for the hypercube $[0,1]^n$ in \Cref{cubeandcubicalhull}. More examples of stabilization and a general theorem 
(in particular for semi-algebraic sets defined by pure binomial inequalities)
are given in \Cref{sec:stab}.

The rest of the paper is organized as follows. In \Cref{sec:background}, we introduce the necessary background in tropical geometry and 
build up the theory of convex sets with the Hadamard property, including moment cones and toric spectrahedra. In \Cref{sec:convexity} we introduce cones of convex and mid-point convex functions on a lattice set $A\subset \Z^n$ and discuss their facet-defining inequalities. 
\Cref{sec:moment} is dedicated to understanding the tropicalization of the moment cones $M_A(S)$. The tropicalizations of the corresponding pseudo-moment cones are discussed in \Cref{sec:SOS}. 
In \Cref{sec:fd} we discuss open question and further research directions.

{\bf Acknowledgements}. 
We would like to thank the anonymous referees for their careful reading and useful suggestions. 
GB is partially supported by US National Science Foundation grant  DMS-1901950. 
CV is partially supported by the US NSF-DMS grant \#1943363. 
 JY is partially supported by US National Science Foundation grant \#1855726. This material is based upon work directly supported by the National Science Foundation Grant No. DMS-1926686, and indirectly supported by the National Science Foundation Grant No. CCF-1900460.


\section{ Tropicalization, Hadamard property, and toric spectrahedra}
\label{sec:background}

\subsection{Tropicalization}\label{subsec:TropDefs}

The tropicalization of a subset $X \subset (\RR^*)^n$ is defined as
the \emph{logarithmic limit set} as in \cite{Alessandrini}:
\[
\trop(X) = \lim_{t \rightarrow \infty} \{(\log_t |x_1|, \dots, \log_t|x_n|) : x \in X\}.
\]  
It is often convenient to characterize it as an image under valuation as we will now explain.
Let $\cR$ be a real closed field with a compatible convex non-trivial non-archimedean valuation
\mbox{$\val : \cR^* \rightarrow \RR$}, such as the field of real
Puiseux series with the order map, as in~\cite{JSY}.  The value group $\Gamma$ (the image of $\val$) is
 dense in $\RR$ since the valuation is nontrivial and the field is real closed.
For a point $x \in \cR^n$,  the {\em tropicalization map} is the 
negation of coordinate-wise valuation
\[
\trop(x) = (-\val(x_1),\dots,-\val(x_n)).
  \]
For any subset $X \subset \cR^n$, we define its {\em tropicalization} to be
\[
\trop(X) := \overline{\{(-\val(x_1),\dots,-\val(x_n))  : x \in X \cap (\cR^*)^n\}} \subset \R^n
\]
where the closure is taken in Euclidean topology of $\R^n$.
By the proof of Lemma~6.4 in~\cite{JSY}, we have
\begin{equation}
\label{eqn:closure}
\trop(X) = \trop(\overline{X})
\end{equation}

for any subset $X \subset \cR^n$ where the closure is taken under the non-archimedean norm.
If $X$ is semialgebraic, $\trop(X)$ is a closed polyhedral subset of $\RR^n$~\cite{tropSpec}. 

Let us now return to a semialgebraic subset $X \subset \RR^n$.  Let
$X_\cR$ be the semialgebraic subset of $\cR^n$ defined by the same
semialgebraic expression
(which, to be precise, means a first-order formula in the language of ordered rings)
defining $X$, where $\cR$ is a real closed
field extension of $\RR$ with a non-trivial valuation.  We then define
 tropicalization to be
\[\trop(X) := \trop(X_\cR).\]
The
tropicalization does not depend on the choice of the extension $\cR$
or the choice of semialgebraic expression defining $X$, and it
coincides with the logarithmic limit set~\cite{JSY, Alessandrini}.
  
We now discuss tropicalizing polynomial inequalities. 
Consider the {\em tropical semiring} $(\RR, \oplus, \odot)$  where the tropical addition $\oplus$ is taking maximum and the tropical multiplication $\odot$ is usual addition.
For a polynomial $f \in \cR[x_1,\dots,x_n]$, let $\trop(f)$ be the tropical polynomial obtained by replacing addition and multiplication with $\oplus$ and $\odot$ respectively and replacing the coefficients with the negative of their valuations.
For instance, for $f = (7\varepsilon+2\varepsilon^3) x_1^2 - (\pi \varepsilon^{-1} + 5) x_1x_2 + 3 \in  \R\{\!\{\varepsilon\}\!\}[x_1,x_2]$, we get $\trop(f) = ((-1)\odot x_1\odot x_1 ) \oplus (1\odot x_1\odot x_2) \oplus 0$, which is the same as $\max\{ -1+2x_1, 1 + x_1 + x_2,0\}$.

Let us define the following notations:
\begin{align*}
  \{\trop(f) \geq 0\} & = \{x \in \RR^n:  \text{ maximum in } \trop(f)
                        \text{ is attained at a positive term} \}\\
    \{\trop(f) > 0\} & = \{x \in \RR^n:  \text{ maximum in } \trop(f) \text{ is attained only at positive term(s)} \}.
  \end{align*}
In the above example, the terms $-1+2x_1$ and $0$ are positive (as the tropicalizations of terms with positive coefficients) and $1 + x_1 + x_2$ is negative so that $\trop(f)\geq 0$ is described by the condition that $-1+2x_1 \geq 1+x_1+x_2$ or $0\geq 1+x_1+x_2$. 
By the definition of tropicalization, for $x \in \cR_{>0}^n$
  \begin{equation}
\label{eqn:weakineq}
    f(x) \geq 0 \implies \trop(x) \in \{\trop(f) \geq 0\}.
    \end{equation}
  Taking contrapositive and changing signs of $f$, we get
  \begin{equation}
\label{eqn:strictineq}
    \trop(x) \in \{\trop(f) > 0\}  \implies f(x) > 0.
    \end{equation}

    For any semialgebraic set $X \subset \RR_{>0}^n$, the {\em
      Fundamental Theorem}~\cite[Theorem 6.9]{JSY} implies that
    \[
\trop(X) = \bigcap_{\stackrel{f \geq 0}{\text{on } X}} \{\trop(f) \geq
0\}
\]
where the intersection can be taken to be finite.  Conversely, any
tropical polynomial inequality valid on $\trop(X)$ arises as the tropicalization
of a polynomial inequality over $\cR$ valid on $X$ \cite[Lemma 6.8]{JSY}.

For semialgebraic sets $X_1$ and $X_2$ in $(\RR^*)^n$ or $(\cR^*)^n$,
we have $\trop(X_1 \cup X_2) = \trop(X_1) \cup \trop(X_2)$ by
definition.  For the intersections we have from \cite[Proposition 6.12]{JSY}, \cite[Lemma 2.3]{tropSpec}:
  \[
    \topint \left(\trop(X_1) \cap \trop(X_2)\right) \subset \trop(X_1 \cap X_2) \subset \trop(X_1) \cap \trop(X_2). 
  \]

It follows that, for a semialgebraic set $X\subset (\R^*)^n$ defined by inequalities $f_1 \geq 0, \dots, f_r \geq 0$, if the set $T = \{x\in \R^n\colon \trop(f_1)(x)\geq 0, \ldots, \trop(f_r)(x)\geq 0\}$ has regular support (meaning that it is equal to the closure of its interior), then $\trop(X) = T$ \cite[Corollary 4.8]{tropSpec}.

If, in particular, the semialgebraic set $X$ is defined by binomial inequalities and their tropicalization, which are usual linear inequalities, cut out a full-dimensional set, then that polyhedral cone coincides with $\trop(X)$. 

\subsection{Tropical Convexity}
\label{sec:tropConv}
  
Now we recall some basics of tropical convexity.  
The \emph{tropical conical hull} of a set $S\subset \R^n$ is the set of
all tropical linear combinations of points in $S$~\cite{DevelinSturmfels},
\begin{equation}
\tcone(S) = \left\{(c_1 \odot \mathbf{s}_1) \oplus ( c_2 \odot \mathbf{s}_2) \oplus \cdots \oplus   (c_r \odot \mathbf{s}_r)  :  r\in \N, \mathbf{s}_1, \dots, \mathbf{s}_r \in S, c_1,\dots,c_r \in \R \right\}. 
\end{equation}
Here $c\odot \mathbf{s} = c \odot (s_1,\ldots,s_n)$ is the vector $(s_1 + c,s_2 + c,\ldots,s_n+c)$.
A subset $S \subset \R^n$ is called \emph{tropically convex} if it equals its tropical conical hull. 
For semialgebraic subsets in the positive orthant, the operations of tropicalization and conical hull commute~\cite[Lemma 8]{tropSA}.  That is,
for any semialgebraic subset $S\subset \R_{> 0}^n$, the tropicalization of the conical hull of $S$ equals the tropical conical hull of $\trop(S)$. That is,  
\begin{equation}
\label{eqn:TropConvCommuting}
\trop(\cone(S)) = \tcone(\trop(S)).
\end{equation}

This follows from the definition of $\trop(S)$ as the image of the set $ S_{\cR}$ 
under coordinate-wise valuation, as described in \Cref{subsec:TropDefs}, and 
the fact that addition is compatible with tropicalization for elements in $\cR_{>0}$ as no cancellation can occur. 
That is, for any $s, t\in \cR_{>0}$, $-\val(s+t) = -\val(s)\oplus -\val(t)$. 
For any points $x, y\in \cR_{>0}^n$ and scalars 
$\lambda, \mu \in \cR_{>0}$
we have  $-\val(\lambda x + \mu y) = (-\val(\lambda) \odot -\val(x))\oplus  (-\val(\mu) \odot -\val(y))$. Therefore
for any set $S_{\cR}\subset \cR_{>0}^n$, the tropical conical hull of $\trop(S_{\cR})$ 
coincides with the tropicalization of its conical hull. 
For any semialgebraic convex cone $S \subset \cR_{>0}^n$, $\trop(S_{\cR})$ is thus 
already tropically convex.

Moreover, the tropical conical hull of a convex polyhedron is again a convex polyhedron. This follows from the following description \cite{HLS, LohoSmith}: Let $\mathbbm{1} = (1,1,\dots,1)$ and 
\[V_i = \{-x \in \R^n : x_i = x_1\oplus \cdots \oplus x_n \} = \cone\{e_j : j \neq i\}+\mathop{span} \mathbbm{1}.\]
For any $Y \subset \R^n$, its tropical conical hull is
\begin{equation}\label{eqn:tcone}
\tcone(Y) = \bigcap_{i=1}^n (Y + V_i).
\end{equation}

We apply this formula to the case where $Y$ is a convex cone. 

\begin{proposition}\label{prop:tconvdual}
Let $Y\subset \mathbb{R}^n$ be a convex cone. Then $\tcone(Y)$ is a convex cone and its dual cone has the form
\[\tcone(Y)^\vee= \sum_{i=1}^n (Y^\vee \cap U_i\cap H),
\]
where $\sum$ is Minkowski addition, $U_i$ is the orthant of $\RR^n$
where the $i$-th coordinate is nonpositive and the rest are
nonnegative, and $H$ is the hyperplane in $\R^n$ perpendicular to $\mathbbm{1}$. In particular, any extreme ray of the dual cone $\tcone(Y)^\vee$ has the form $\sum \alpha_ie_i$, where $e_i$ are the standard basis vectors, $\sum_{i=1}^n \alpha_i=0$, and exactly one of the coefficients $\alpha_i$ is negative.
\end{proposition}

\begin{proof}
We have $V_i=U_i+\operatorname{span} \mathbbm{1}$.
By convex duality, it follows that $V_i^\vee=U_i^\vee\cap H$. Moreover, $U_i^\vee = U_i$ because $U_i$ is an orthant.
By~\eqref{eqn:tcone}, we have $\tcone(Y)^\vee= \sum_{i=1}^n (Y^\vee \cap U_i\cap H)$.
In particular, any extreme ray of $\tcone(Y)^\vee$ lies in $Y^\vee \cap V_i^\vee$, and the proposition follows.
\end{proof}

\subsection{Hadamard Property}
We say that a subset of $\mathbb{R}^n$ has the {\em Hadamard property}
if it is closed under Hadamard (coordinate-wise)
multiplication
\[
\RR^n \times \RR^n \rightarrow \RR^n, ~~~~ (x,y) \mapsto x\circ y := (x_1 y_1, \dots, x_n y_n).
  \]
The closure of a set with Hadamard property also has Hadamard property.

\begin{proposition}\label{prop:had}
Let $A\subset \mathbb{Z}^n_{\geq 0}$ be a finite set of nonnegative lattice points and let $\varphi_A: \mathbb{R}^n\rightarrow \mathbb{R}^{A}$, $x \mapsto (x^a \colon a\in A)$, be the corresponding monomial map.
If $S\subseteq \mathbb{R}^n$ has the Hadamard property then so do
\begin{enumerate}
\item the image $\varphi_A(S)$ of $S$ under $\varphi_A$,
\item the convex hull $\conv(S)$ and the conical hull $\cone(S)$, and 
\item the closed moment cone $M_A(S)$.
\end{enumerate}
\end{proposition}

The moment cone is by definition the cone of all moment sequences $(m_{\alpha}\colon \alpha \in A)$, that is $m_\alpha$ is the integral $\int_S x^\alpha d\mu$ of $x^\alpha$ over the set $S$ with respect to some measure $\mu$. In general, this cone need not be closed. We denote its closure by $M_A(S)$ and call it the closed moment cone. When $A$ is finite, the closed moment cone coincides with the closed conical hull of $\varphi_A(S)$:
\[
  M_A(S) = \overline{\cone \varphi_A(S)}.
\]
See, for example, \cite[Theorems 1.24 and 1.26]{SchmudgenBook}.

\begin{proof}[Proof of \Cref{prop:had}]
For (1), note that for $x,y\in S$, $\varphi_A(x)\circ\varphi_A(y) = \varphi_A(x\circ y)$ belongs to $\varphi_A(S)$. 

For (2), let $x_1, \hdots, x_\ell, y_1, \hdots, y_m \in S$ and $\lambda_1, \hdots, \lambda_{\ell}, \mu_1, \hdots, \mu_m\in \R_{\geq 0}$. Then 
\[
\left(\sum_{i=1}^{\ell} \lambda_i x_i\right)\circ  \left(\sum_{j=1}^m \mu_j y_j\right)
= 
\sum_{i=1}^{\ell}\sum_{j=1}^m \lambda_i\mu_j (x_i  \circ y_j).
\]
This shows that the conical hull $\cone(S)$ also has the Hadamard property.  
Moreover, if $\sum_i \lambda_i = 1$ and $\sum_j \mu_j=1$, 
then $\sum_{i=1}^{\ell}\sum_{j=1}^m \lambda_i\mu_j  = 1$. 
The expression above then shows that the convex hull of $S$ also has the Hadamard property. 

For (3), as $M_A(S) = \overline{\cone
(\varphi_A(S))}$, it has the Hadamard property by parts (1) and (2).
\end{proof}

\begin{example}
A famous example of a set with Hadamard property is the cone of positive semidefinite $n \times n$ matrices $\mathcal{S}^n_+$, which is closed under Hadamard products by the Schur product theorem. We can recover the Schur product theorem from \Cref{prop:had} as follows. The cone $\mathcal{S}^n_+$ is the convex hull of rank one matrices \cite[Chapter II.12]{MR1940576} and the set $\{xx^t \colon x\in \R^n\}$ of rank $1$ matrices is the image of $\mathbb{R}^n$ under the monomial map $\varphi_A$  where $A\subset \mathbb{Z}^n_{\geq 0}$ is the set of all vectors in $\mathbb{Z}^n_{\geq 0}$ where the sum of coordinates is two.
\end{example}

For any semialgebraic set $S$, its tropicalization is a rational polyhedral fan. If, additionally, $S$ has the Hadamard property, then its tropicalization is closed under addition, so it is a rational convex polyhedral cone. By Proposition~\ref{prop:had}, $M_A(S)$ also has Hadamard property, so $\trop(M_A(S))$ is also a rational polyhedral cone.

One motivation for studying the tropicalization of a set $X$ is to study the 
set of pure binomial inequalities that are valid on $X$.  This is especially true when $X$ has the Hadamard property as the following two statements show. 

\begin{proposition}\label{prop:TropConeLog}
For any set $X\subset \R_{>0}^n$, $\trop(X)$ is contained in $\overline{\cone(\log(X))}$, with equality if $X$ has the Hadamard property.
\end{proposition}

The proposition still holds if the word {\em cone} is taken to mean just
  multiplying by positive constants, rather than conical hull.

\begin{proof}
Fix any base greater than one for $\log$. 
The tropicalization of $X$ can be written as the pointwise limit as 
$t\to \infty$ of the set $\frac{1}{t}\log(X)$. 
For each $t>0$, $\frac{1}{t}\log(X)$ is contained in the cone 
over $\log(X)$ and so the limit as $t\to \infty$ is contained in its closure, 
$\overline{\cone(\log(X))}$. 
Equality in the case that $X$ has the Hadamard property follows from \cite[Lemma 2.2]{BRST}.
\end{proof}

\begin{proposition}\label{prop:binomIneq}
Let $X\subseteq \R_{>0}^n$ and $\alpha = {\alpha_+}  - {\alpha_-}$ where ${\alpha_+}, {\alpha_-} \in \Z_{\geq 0}^n$. 
The linear inequality $\sum_{i=1}^n \alpha_i x_i\geq 0$ holds on $\cone(\log(X))$ 
if and only if the binomial inequality $x^{\alpha_+} \geq x^{\alpha_-}$ holds on $X$.  \end{proposition}
\begin{proof}
First, note that a linear inequality holds on $\cone(\log(X))$ if and only if it 
holds on $\log(X)$.  Consider a point $x\in X$. The inequality 
 $x^{\alpha_+} \geq x^{\alpha_-}$ is equivalent to the Laurent inequality $x^{\alpha}\geq 1$ since $X\subset \R_{>0}^n$. 
 Because $\log$ is monotonic, $x^{\alpha}\geq 1$ if and only if 
 $\sum_{i=1}^n \alpha_i \log(x_i) \geq 0$. 
\end{proof}

\begin{remark}
  \label{rmk:pureBinomials}
Together, these propositions show that if $X\subseteq \R_{>0}^n$ 
has the Hadamard property, then $\trop(X)$ is a convex cone 
whose \emph{dual cone} consists of the set of pure binomial inequalities that
are valid on $X$. 
We need to consider only pure binomials in the presence of Hadamard
property because they are the strongest possible binomial inequalities.  
An inequality of the form $x^{\alpha_+} \geq c x^{\alpha_-}$ holds on $X$ if and only 
if $\sum_{i=1}^n \alpha_i \log(x_i) \geq \log(c)$ holds on $\log(X)$.  
Since $\trop(X) = \overline{\cone(\log(X))}$ is a convex cone, 
for any affine-linear inequality $\sum_{i=1}^n \alpha_i y_i \geq a_0$ that holds on 
$\trop(X)$, the (stronger) linear inequality  $\sum_{i=1}^n \alpha_i y_i \geq 0$ also holds. 
\end{remark}

We are particularly interested in semi-algebraic subsets of the nonnegative orthant given by pure binomial inequalities. These always have the Hadamard property.
\begin{lemma}\label{lem:hadamardForS}
  Let $S\subset \R_{\geq 0}^n$ be a semi-algebraic set given by the inequalities $x_i\geq 0$ ($i=1,2,\ldots,n$) and pure binomial inequalities $g_j = x^{a_j} -x^{b_j} \geq 0$ ($j=1,2,\ldots,r$). Then $S$ has the Hadamard property.
\end{lemma}

\begin{proof}
  Let $x,y \in S$.
  Then every entry of the Hadamard product $x\circ y$ is already nonnegative, so it remains to check the binomial inequalities $g_j \geq 0$. For this, we compute 
  \[ (x\circ y)^a - (x\circ y)^b = (x^a-x^b)\cdot y^a + (y^a - y^b)\cdot x^b \]
  which shows that every binomial inequality $g_j \geq 0$ defining $S$ also holds for the point $x\circ y$.
\end{proof}

\subsection{Tropicalizations of Toric Spectrahedra} 
\label{subsec:toricspecs}
In this section, we consider toric spectrahedra and some generalizations relevant for pseudo-moment cones in the pure binomial case (see \Cref{subsec:pseudoPureBinomial}).
A \emph{toric spectrahedron} is a semialgebraic subset of $\mathbb{R}^n_{\geq 0}$ defined by the positive-semidefiniteness of a matrix with monomial (rather than linear) entries in a set of variables.   

Formally, let $A(x)$ be a symmetric matrix with entries that are monomials in variables $x=(x_1,\dots,x_n)$ and let $X$ be the subset of $\mathbb{R}^n_{\geq 0}$ consisting of the points $x$ which make matrix $A(x)$ positive semidefinite. 
By the Schur product theorem, every toric spectrahedron has the Hadamard property. 

In the special case when each entry of $A(x)$ is one of the variables, so a monomial of degree $1$, the toric spectrahedron $X$ is a spectrahedron in the usual sense. An important example is the Hankel spectrahedron, the dual convex cone to the cone of sums of squares.

Let $X = \{x\in \R_{\geq 0}^n \colon A(x)\succeq 0\}$ be a toric spectrahedron.  Let $X'$ be the subset of $\mathbb{R}^n_{\geq 0}$ given by the values of $x$ which make all $2\times 2$ principal minors of $A(x)$ nonnegative. Clearly $X\subset X'$. Since the $2\times 2$ principal minors of $A(x)$ are pure binomials, $\trop(X')$ is a polyhedral cone given by tropicalization of $2\times 2$ principal minors of $A(x)$.  If $\trop (X')$ has non-empty interior in $\mathbb{R}^n$ and
no $2 \times 2$ principal minor of $A(x)$ is identically zero, then we in fact have, from \cite[Theorem 4.4]{BRST},
\begin{equation}
\label{eqn:monom}
\trop (X)=\trop (X').\end{equation}
That is, the tropicalization of toric spectrahedron of $A(x)$ is defined by tropicalization of $2\times 2$ minors of $A(x)$.


\begin{remark} One can broaden the definition of toric spectrahedra to allow $X$ to not be contained in the nonnegative orthant.  However in this case we can replace $X$ with its image $|X|$ under coordinate-wise absolute value. 
By the Schur product theorem, $X$ has Hadamard property, and so does $|X|$. 
Then $\trop(|X|) = \trop(|X'|)$ and we can obtain all pure binomial inequalities in absolute values that are valid on $X$.\end{remark}

We can use similar ideas to deal with sets defined by differences of monomial matrices being positive semidefinite. 

\begin{proposition}\label{prop:binom}
Let $A(x)$ and $B(x)$ be symmetric matrices whose entries are pure monomials in the variables $x=(x_1,\dots,x_n)$. Let $X$ be the subset of $\mathbb{R}^n_{\geq 0}$ consisting of points $x$ such that $A\succeq B\succeq 0$, or equivalently, $A-B\succeq 0$ and $B \succeq 0$. 
Let $X'$ be the subset of $\mathbb{R}^n_{\geq 0}$ given by the values of $x$ which make all $2\times 2$ principal minors of $A$ and $B$ nonnegative, and the diagonal entries of $A-B$ nonnegative. 
\begin{enumerate}
  \item The set $X$ has the Hadamard property.
  \item If $\trop (X')$ has non-empty interior in $\mathbb{R}^n$ and no $2 \times 2$ principal minor of $A$ or $B$ is identically zero, then $\trop (X)=\trop (X')$ and both are given by tropicalization of $2\times 2$ minors of $A$ and $B$ and diagonal entries of $A-B$.
\end{enumerate}
\end{proposition}

\begin{proof}
To show claim (1), suppose that $x,y \in \mathbb{R}^n$ such that $A(x)\succeq B(x)\succeq 0$ and $A(y)\succeq B(y)\succeq 0$. Then $B(x\circ y)=B(x)\circ B(y)\succeq 0$ by the Schur product theorem. Also by the Schur product theorem 
\[ A(x\circ y)-B(x\circ y)=A(x)\circ (A(y)-B(y))+(A(x)-B(x))\circ B(y)\succeq 0,\] 
and therefore $A(x\circ y)\succeq B(x\circ y)$ showing that $X$ has Hadamard property.

We now show claim (2). Since the $2\times 2$ minors of $A$ and $B$ as well as the diagonal entries of $A-B$ are pure binomials, the set $X'$ has the Hadamard property, see \Cref{lem:hadamardForS}.
So Propositions \ref{prop:TropConeLog} and \ref{prop:binomIneq} imply that $\trop(X')$ is given by tropicalization of $2\times 2$ minors of $A$ and $B$ as well as the diagonal entries of $A-B$.
So we only need to show $\trop(X) = \trop(X')$. The inclusion $\trop(X)\subseteq \trop(X')$ is immediate because $X\subseteq X'$.
To show the other inclusion $\trop(X')\subseteq \trop(X)$, it suffices to show that the interior of $\trop(X')$ is contained in $\trop(X)$, because $\trop (X')$ is full-dimensional and $\trop(X)$ is closed. This is a consequence of the following \Cref{lem:manyBinom} below, with $r=1$, $A_0(x) = A(x)$, $A_1(x) = B(x)$ and $c_1 = -1$. 
\end{proof}

Consider the polynomial matrix
\[
A(x) = A_0(x) + \sum_{k=1}^r c_k A_k(x)
\]
where $c_1, \hdots, c_r\in \R$ and each entry of $A_k(x)$ is a monomial in $x$. 
Specifically, suppose that the $(i,j)$th entry of $A_k(x)$ is $x^{\beta}$ where $\beta = \alpha_{ij}^k$. 
Let $S$ denote the set of points $x\in \R_{\geq 0}^n$ for which $A(x)\succeq 0$. 

\begin{lemma}\label{lem:manyBinom}
Let $y\in \R^n$ be a vector strictly satisfying the tropical inequalities given by 
the positivity of the $2\times 2$ minors of $A_k$ and the diagonal entries of $A_0 - A_k$. 
That is,
\begin{align*}
y\cdot \alpha_{ii}^0 > y\cdot \alpha_{ii}^k  & \ \  \text{ for all $i$ and all $k\neq 0$} \\
y\cdot \alpha_{ii}^k + y\cdot \alpha_{jj}^k  > 2 y\cdot \alpha_{ij}^k  & \ \  \text{ for all $i\neq j$ and all $k$}.
\end{align*}
Then for all sufficiently large $t\in \R$, $A(t^y)$ is positive definite and $y$ belongs to $\trop(S)$. Here, $t^y$ denotes the vector $(t^{y_1},\ldots,t^{y_n})$.
\end{lemma}
\begin{proof}
Showing that $A(t^y)$ is positive definite for all sufficiently large $t$ implies that $y$ belongs to $\trop(S)$. 
We do this by expanding the principal minors of $A(t^y)$. 
Because the entries of $A_k(x)$ are monomial in $x$, the entries of 
$A_k(t^y)$ are powers of $t$. Specifically, the $(i,j)$th entry of $A_k(t^y)$ is $t^b$ where $b = y\cdot \alpha_{ij}^k$. 

Consider the Laplace expansion of the determinant of $A(t^y)$, expanded as an exponential polynomial in $t$. 
We claim that $\sum_{i}y\cdot \alpha_{ii}^0$ will be the leading exponent  of $t$ that will be obtained uniquely by the product of the diagonal of $A_0$. 

To see this, note that for all $k$
\[
2y\cdot \alpha_{ij}^k < y\cdot \alpha_{ii}^k + y\cdot \alpha_{jj}^k\leq  y\cdot \alpha_{ii}^0 + y\cdot \alpha_{jj}^0
\]
where the second inequality is an equality only when $k=0$. 
Formally, consider the exponent of $t$ in the Laplace expansion of the determinant 
of $A(t^y)$ given by the permutation $\pi$ of $[d]$ and a choice $\sigma:[d]\to \{0,1,\hdots, r\}$ 
of one of the $r+1$ terms in the $(i,\pi(i))$th entry of $A(t^y)$. 
The resulting exponent of $t$ is  
\[
\sum_i y\cdot \alpha_{i\pi(i)}^{\sigma(i)}\leq  \frac{1}{2}\sum_i y\cdot \alpha_{ii}^{\sigma(i)}+ y\cdot \alpha_{\pi(i)\pi(i)}^{\sigma(i)}
\leq \sum_i y\cdot \alpha_{ii}^0.
\]
Moreover the equality is possible for the left and right hand sides only when $\pi$ is the identity permutation and $\sigma(i) = 0$ for all $i$. 

This shows that the Laplace expansion of the determinant of $A(t^y)$ has the form $t^{b}$, where $b=\sum_{i}y\cdot \alpha_{ii}^0$, plus terms of strictly lower degree in $t$. For large enough $t\in \R$, the determinant of $A(t^y)$ will therefore be positive.  The same argument can be made for all the principal minors of $A(t^y)$, showing that $A(t^y)$ is positive definite for large enough $t$. 
\end{proof}


\section{Convexity vs mid-point convexity}
\label{sec:convexity}

As we will see in Sections \ref{sec:tropmom} and \ref{sec:tsos}, the tropicalizations of the moment cone and the dual cone to the cone of sums of squares relate to two different notions of convexity for real functions on lattice points. 

\begin{definition}
Let $A \subset \ZZ^n$ be finite. We say that a function $h : A \rightarrow \RR$ is {\em convex} if $\sum \lambda_i h(a_i) \geq h(b)$ whenever $\sum \lambda_ia_i = b$ with $a_i,b\in A$, $\lambda_i\geq 0$, and $\sum \lambda_i=1$. 
We say that $h$ is {\em midpoint convex} if $\frac{h(a_1)+h(a_2)}{2} \geq h(\frac{a_1+a_2}{2})$ whenever $a_1,a_2,\frac{a_1+a_2}{2} \in A$. 
\end{definition}

As we explain in Sections \ref{sec:tropmom} and \ref{sec:tsos}, convex functions correspond to tropicalizations of moment cones and mid-point convex functions correspond to tropicalizations of pseudo-moment cones. 

Every convex function is mid-point convex, but the converse depends on the geometry of the set $A$. 
We denote
$$\mathcal K_A := \{ h \in \RR^A : h \text{ is convex}\} 
\quad \subseteq \quad
\mathcal M_A := \{ h \in \RR^A : h \text{ is midpoint convex}\}.$$
The sets $\mathcal K_A$ and $\mathcal M_A$ are defined by linear inequalities in $\RR^A$ and thus are convex cones. 
The description of $\mathcal M_A$ above involves only finitely many inequalities, and thus $\mathcal M_A$ is clearly a polyhedral cone.
The set $\mathcal K_A$ is also a polyhedral cone, as we explain below.
We also present the non-redundant inequality descriptions of these two polyhedral cones, and investigate exactly when they differ. 

Any function $h : A \rightarrow \RR$ induces a {\em marked regular subdivision} $\mathcal S_h$ of the point configuration $A$, by projecting back to $\RR^A$ the lower faces of the ``lifted polytope'' $P_h = \conv\{(a,h(a)) \in \RR^A \times \RR : a \in A \}$. A point $a \in A$ is marked in $\mathcal S_h$ if the corresponding point $(a,h(a))$ lies in a lower face of $P_h$. 
From this point of view, convex functions on $A$ are precisely the functions that induce a regular subdivision with all points marked.  See \cite{triangulations} for precise definitions.
Convex functions $h$ on $A$ are also precisely those that arise as the restriction to $A$ of a {maximum} of affine linear functions on $\RR^n$ (in other words, the tropical sum of affine linear functions on $\R^n$).

\begin{definition}
A subset $T = \{a_0, \dots, a_k, b\} \subset A$ with $k+2$ elements is called a $k$-dimensional {\em punctured simplex} if $a_0, \dots, a_k$ are affinely independent and $b$ lies in the $k$-dimensional simplex $\Delta = \conv(a_0, \dots, a_k)$, possibly on its boundary. We say that $T$ is an {\em almost-empty simplex} of $A$ if it is punctured and in addition $\Delta \cap A = T$ and $b$ lies in the relative interior of $\Delta$.

We say that a punctured simplex $T = \{a_0, \dots, a_k, b\} \subset A$ {\em certifies the non-convexity} of a function $h:A \to \mathbb R$ if $\sum_{i=0}^k \lambda_i h(a_i) < h(b)$ for scalars $\lambda_i \geq 0$ such that $\sum_{i=0}^k \lambda_i a_i = b$ and $\sum_{i=0}^k \lambda_i = 1$.
\end{definition}

Figure \ref{fig:puncturedsimplices} shows a set $A \subset \ZZ^2$ consisting of $7$ lattice points (marked as circles), and three different punctured simplices in it (corresponding to the filled circles). Only the punctured simplex on the right is an almost-empty simplex of $A$.

\begin{figure}[ht]
  \centering
  \subfloat[][]{
   \includegraphics[scale=1]{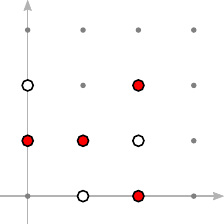}
   \label{fig:puncturedsimplex1}
  }
  \hspace{10mm}
    \subfloat[][]{
   \includegraphics[scale=1]{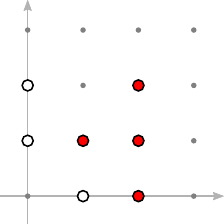}
   \label{fig:puncturedsimplex2}
  }
  \hspace{10mm}
  \subfloat[][]{
    \includegraphics[scale=1]{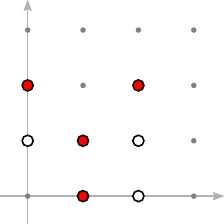}
   \label{fig:puncturedsimplex3}
  }
  \caption{A set of $7$ lattice points $A \subset \ZZ^2$ and three punctured simplices in it. Only the punctured simplex on the right is an almost-empty simplex of $A$.}
  \label{fig:puncturedsimplices}
\end{figure}

The next lemma is the key to providing the (non-redundant) facet description of the cone of convex functions on a finite set $A \subset \ZZ^n$. 

\begin{lemma}\label{lem:emptysimplex}
If $h:A \to \mathbb R$ is not convex then there is an almost-empty simplex $T \subset A$ that certifies its non-convexity.
\end{lemma}

\begin{proof}
We first prove that there is a punctured simplex $T \subset A$ that certifies the non-convexity of $h$. Since $h$ is not convex, there exists $b \in A$ such that $(b,h(b))$ lies strictly above a lower face $\hat F$ of the lifted polytope $P_h = \conv\{(a,h(a)) \in \RR^A \times \RR : a \in A \}$. The face $\hat F$ projects down to a face $F$ in the marked regular subdivision $\mathcal S_h$ induced by $h$. The point $b$ is then in the convex hull of the marked points of $\mathcal S_h$ that lie in $F$, and by Carath\'eodory's Theorem, $b$ is in the convex hull of an affinely independent subset $a_0, \dots, a_k$ of them. The punctured simplex $\{a_0, \dots, a_k, b\} \subset A$ then certifies the non-convexity of $h$.

Now, take a punctured simplex $T = \{a_0, \dots, a_k, b\} \subset A$ that certifies the non-convexity of $h$ such that $\conv(T) \cap A$ has as few elements as possible. Note that $b$ must lie in the interior of the simplex $\Delta := \conv(a_0, \dots, a_k)$, as otherwise the proper face of $\Delta$ where $b$ lies would still certify the non-convexity of $h$ and contain fewer points of $A$. After adding a suitable affine function to $h$, we can assume that $h(a_i) = 0$ for all $i$ and $h(b)>0$.

Assume for contradiction that $T$ is not an almost-empty simplex, so $\Delta \cap A$ contains a point $a$ not in $T$. Suppose first that $h(a) \geq h(b)$. For $i = 0, \dots, k$, consider the $k$-dimensional (closed) affine cone $C_i$ with vertex $b$ and spanned by the rays $a_j - b$ with $j \neq i$. These affine cones $C_0, \dots, C_k$ form a fan that covers the whole affine span of $\Delta$, and so $a \in C_l$ for some $l$. This implies that $a \in \conv(a_0, \dots, \hat{a_l}, \dots, a_k, b)$. But then $T' := (T \setminus \{a_l\}) \cup \{a\}$ is a punctured simplex that certifies the non-convexity of $h$ and
$|\conv(T') \cap A| < |\conv(T) \cap A|$,
which is a contradiction. Now, suppose that $h(a) \leq h(b)$. Let $C'_i$ be the reflection of the cone $C_i$ across its vertex $b$. Again, the collection of cones $C'_0, \dots, C'_k$ cover the affine span of $\Delta$, so $a \in C_{l'}$ for some $l'$. This implies that $b \in \conv(a_0, \dots, \hat{a_{l'}}, \dots, a_k, a)$. But then $T' := (T \setminus \{a_{l'}\}) \cup \{a\}$ is a punctured simplex that certifies the non-convexity of $h$ and
$|\conv(T') \cap A| < |\conv(T) \cap A|$,
which is a contradiction.
\end{proof}

Note that the proof of the previous lemma provides a constructive argument for finding an almost-empty simplex that certifies the non-convexity of a function.

\begin{proposition}\label{p:conesfacets}
The set $\mathcal K_A$ is a polyhedral cone in $\RR^A$. Its facets correspond to inequalities of the form $\sum_{i=0}^k \lambda_i h(a_i) \geq h(b)$ where $\{a_0, \dots, a_k, b\} \subset A$ is an almost-empty simplex of $A$ and $\lambda_i > 0$ are the unique scalars such that $\sum \lambda_i a_i = b$ and $\sum \lambda_i = 1$. 
%
\end{proposition}
We describe the facets of the cone $\mathcal M_A$ later in Proposition \ref{prop:facetsmidpoint}.
\begin{proof}
By Lemma \ref{lem:emptysimplex}, the set $\mathcal K_A$ is the collection of real functions $h$ on $A$ that satisfy the linear inequalities $\sum_{i=0}^k \lambda_i h(a_i) \geq h(b)$ for any almost-empty simplex $\{a_0, \dots, a_k, b\} \subset A$ and the unique scalars $\lambda_i \geq 0$ such that $\sum \lambda_i a_i = b$ and $\sum \lambda_i = 1$. In particular, this implies that $\mathcal K_A$ is a polyhedral cone. To see that all of these are indeed facet-defining inequalities, note that for any almost-empty simplex $T = \{a_0, \dots, a_k, b\} \subset A$ there exists an $h \in \RR^A$ that does not satisfy the inequality corresponding to $T$ but does satisfy all other inequalities. An example of such a function is obtained by taking a convex function $h$ on $A \setminus \{b\}$ satisfying $h(a_i) = 0$ for $i = 0,\dots,k$ and $h(a) \gg 1$ for all other $a \in A \setminus \{a_0,\dots,a_w,b\}$, and setting $0 < h(b) \ll 1$.
%
\end{proof}

The previous proposition allows us to classify the sets $A$ for which midpoint convexity implies convexity.

\begin{corollary}
The inclusion of polyhedral cones $\mathcal K_A \subseteq \mathcal M_A$ is an equality if and only if $A$ does not contain an almost-empty simplex of dimension at least 2 and every $1$-dimensional almost-empty simplex $\{a_0, a_1, b\} \subset A$ satisfies $\frac{a_0 + a_1}{2} = b$.
\end{corollary}

\begin{proof}
By Proposition \ref{p:conesfacets}, if every almost-empty simplex of $A$ is one dimensional and has the form $\{a_0, a_1, \frac{a_0+a_1}{2}\}$ then all the facets of $\mathcal K_A$ are also facets of $\mathcal M_A$, and so $\mathcal K_A = \mathcal M_A$. Conversely, if $A$ contains an almost-empty simplex $T$ of dimension at least 2 or a 1-dimensional almost-empty simplex $T = \{a_0, a_1, b\} \subset A$ with $b \neq (a_0 + a_1)/2$, then a function $h \in \RR^A$ that satisfies all almost-empty simplex inequalities except for the one corresponding to $T$ is an example of a function in $\mathcal M_A$ but not in $\mathcal K_A$.
\end{proof}

\begin{remark}
It is not clear how to simply describe the lineality spaces and the extreme rays of the cones $\mathcal K_A$ and $\mathcal M_A$ for general $A$.
Both cones $\mathcal K_A$ and $\mathcal M_A$ contain the $(n+1)$-dimensional linear subspace of affine linear functions on $A$:
$$L_A = \{ h \in \RR^A :  h(w) = c \cdot w + d, \text{ where } c \in \RR^n, d \in \RR \}.$$
However, their lineality space may be larger. 
For example, if $A = \{0,1\}^2 \subset \RR^2$ the four vertices of a square, then every function on $A$ is convex and mid-point convex.

By definition, the lineality space of $\mathcal K_A$ consists of all the functions $h$ on $A$ such that both $h$ and $-h$ are convex, which means that for every $a \in A$ the lifted point $(a, h(a))$  is in both the upper hull and the lower hull of $P_h = \conv\{(a, h(a)) : a \in A\}$.  In particular, if $A$ contains a point in the interior of its convex hull, then the lineality space of $\mathcal K_A$ is equal to $L_A$.  More generally if $A$ contains a point in the relative interior of a face of $\conv(A)$, then any $h$ in the lineality space of $\mathcal K_A$ must be an affine linear function on that face.  At the other extreme, if every point of $A$ lies in a
simplicial face of $\conv(A)$, then every vertex can be lifted
independently and after that the lifts of other points are determined
uniquely by affinely interpolating, so the lineality space of $\mathcal K_A$ consists of piecewise linear functions induced by arbitrary functions on vertices of $A$.
\end{remark}

We now investigate restrictions of convex and midpoint convex functions to a smaller domain. Let $A \subset E \subset \ZZ^n$ be finite subsets. We consider the natural projection 
$$\pi_{A} : \RR^E \to \RR^A.$$

\begin{proposition}\label{prop:projection}
Suppose $A \subset E \subset \ZZ^n$ are finite subsets. We have
$$\pi_{A}(\mathcal K_E) = \mathcal K_A.$$
If, in addition, $A = \conv(A) \cap \ZZ^n$, we also have
$$\pi_{A}(\mathcal M_E) = \mathcal M_A.$$
\end{proposition}
\begin{proof}
The fact that $\pi_{A}(\mathcal K_E) = \mathcal K_A$ expresses the fact that any convex function $h \in \mathcal K_A$ can be extended to a convex function $\hat h \in \mathcal K_E$, for instance by defining $\hat h : E \to \RR$ to be the height of the lower convex hull of the lifted points $\{(a,h(a)): a \in A\} \subset \RR^n \times \RR$. 
If $A = \conv(A) \cap \ZZ^n$, we can also extend any midpoint-convex function $h \in \mathcal M_A$ to a midpoint-convex function $\hat h \in \mathcal M_E$. For instance, fix an ordering of the set $E \setminus A = \{e_1,\dots,e_m\}$ such that $\conv(A \cup \{e_1, \dots ,e_i\}) \cap E = A \cup \{e_1, \dots ,e_i\}$ for all $1 \leq i \leq m$. Then, recursively for $i = 1, \dots,m$, define $\hat h(e_i) \in \RR$ to be sufficiently large so that $\hat h(e_i) \geq 2 \hat h(e) - h(e')$ for any $e, e' \in A \cup \{e_1, \dots ,e_{i-1}\}$. Since $e_i$ cannot be the midpoint of any two points $e, e' \in A \cup \{e_1, \dots ,e_{i-1}\}$, this defines a midpoint-convex function $\hat h \in \mathcal M_E$ extending $h$.
\end{proof}

Let $A \subset \ZZ^n$, and take $E = \conv(A) \cap \ZZ^n$.
We are interested in understanding when the containment $\mathcal K_A \subset \pi_{A}(\mathcal M_E)$ of polyhedral cones is an equality.

Given any $S \subset \ZZ^n$, its set of \emph{midpoints} is
$$\Mid(S) = \{ (s+t)/2 \in \ZZ^n : s,t \in S \text{ and } s \neq t \}.$$
Suppose $V = \{v_0, \dots, v_k\} \subset \ZZ^n$ is a set of affinely independent lattice points,
and let $F = \conv(V) \cap \ZZ^n$. 
A subset $S \subset F$ is called \emph{$V$-mediated} if $S \supset V$ and 
$S \setminus V \subset \Mid(S)$.
As $V$-mediated subsets are closed under union, there is a maximal $V$-mediated subset of $F$, 
which we denote by $V^*$. It can be computed by starting with the set $F$ and repeatedly
iterating the map $X \mapsto V \cup \Mid(X)$, i.e., iteratively discarding points that are neither in $V$ nor are midpoints of other points in the set. 
It was shown in \cite[Theorem (2.2)]{Reznick89} that this procedure stabilizes (because it is a decreasing sequence of finite sets) and the fix point is the maximal $V$-mediated set $V^*$. 

\begin{example}
If $V_1 = \{ (0,0), (1,2), (2,1) \}$ then the maximal $V_1$-mediated set $V_1^*$ is equal to $V_1$. 
The point $(1,1)$ is in $\conv(V_1) \cap \ZZ^n$, but not in $V_1^*$.
On the other hand, if $V_2 = \{ (0,3), (1,0), (3,1) \}$ then the maximal $V_2$-mediated set $V_2^*$ is equal to the set of six points in $\conv(V_2) \cap \ZZ^n$; see Figure \ref{fig:pointconfigurations}.
\begin{figure}[ht]
  \centering
  \subfloat[][]{
   \includegraphics[scale=1.2]{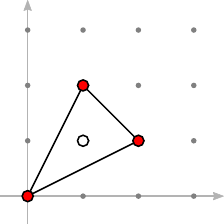}
   \label{fig:motzkintriangle}
  }
  \hspace{30mm}
  \subfloat[][]{
    \includegraphics[scale=1.2]{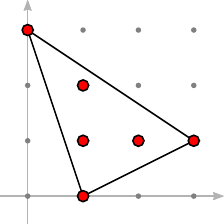}
   \label{fig:sixpointstriangle}
  }
  \caption{Two lattice triangles. Filled in red are the points in their corresponding maximal mediated sets.}
  \label{fig:pointconfigurations}
\end{figure}
\end{example}

Our interest in $V$-mediated sets comes from the following fact.
\begin{lemma}[{\cite{Reznick89}}]\label{l:reznick}
Suppose $V = \{v_0, \dots, v_k\} \subset \ZZ^n$ is a set of affinely independent points,
and let $F = \conv(V) \cap \ZZ^n$. Assume $w \in F \setminus V$, and let $T = V \cup \{w\}$. 
Then $\pi_{T}(\mathcal M_F)$ is equal to the halfspace $\mathcal K_T \subset \RR^T$ if and only if 
$w$ belongs to the maximal $V$-mediated subset $V^*$.
\end{lemma}

\begin{corollary}\label{cor:projequality}
Let $A \subset \ZZ^n$ be finite, and $E \supseteq \conv(A) \cap \ZZ^n$. 
Then $\mathcal K_A = \pi_{A}(\mathcal M_E)$ if and only if 
for every almost-empty simplex $T = \{a_0, \dots, a_k, b\} \subset A$ with vertices $a_0,\dots, a_k$, the relative interior point $b$ belongs to the 
maximal $V$-mediated subset $V^*$, where $V = \{a_0, \dots, a_k\}$.
\end{corollary}
\begin{proof}
Suppose $\mathcal K_A = \pi_{A}(\mathcal M_E)$.
Fix an almost-empty simplex $T = \{a_0, \dots, a_k, b\} \subset A$, and let $F = \conv(T) \cap \ZZ^n$. 
By Proposition \ref{prop:projection}, we have $\pi_{F}(\mathcal M_E) = \mathcal M_F$.
It follows that $\pi_{T}(\mathcal M_F) = \pi_{T}(\mathcal M_E) = \pi_{T}(\pi_{A}(\mathcal M_E)) = 
\pi_{T}(\mathcal K_A) = \mathcal K_T$,
so by Lemma \ref{l:reznick} we have that $b$ belongs to the 
maximal $V$-mediated subset $V^*$, where $V = \{a_0, \dots, a_k\}$.

Conversely, assume that for every almost-empty simplex $T = \{a_0, \dots, a_k, b\} \subset A$, 
the point $b$ belongs to the maximal $V$-mediated subset $V^*$, where $V = \{a_0, \dots, a_k\}$.
By Lemma \ref{l:reznick}, for every almost-empty simplex $T \subset A$, the set 
$\pi_{T}(\mathcal M_F)$ is equal to $\mathcal K_T$, where $F = \conv(T) \cap \ZZ^n$. 
By Proposition \ref{prop:projection}, we then have $\mathcal K_T = \pi_{T}(\mathcal M_F) = \pi_{T}(\pi_{F}(\mathcal M_E)) = \pi_{T}(\pi_{A}(\mathcal M_E))$, 
and thus $\pi_{A}(\mathcal M_E)$ is contained in the halfspace 
$\mathcal K_T \times \RR^{A \setminus T} \subset \RR^A$.
By Proposition \ref{p:conesfacets}, the cone $\mathcal K_A$ is equal to the intersection 
of these halfspaces, and thus $\pi_{A}(\mathcal M_E) = \mathcal K_A$.
\end{proof}

The proof of Lemma \ref{lem:emptysimplex} given in \cite{Reznick89} generalizes directly 
to provide the non-redundant inequality description of the cone $\mathcal M_A$ of midpoint-convex functions.

\begin{proposition}\label{prop:facetsmidpoint}
The facets of the polyhedral cone $\mathcal M_A$ correspond to inequalities of the form $\frac{h(a_1) + h(a_2)}{2} \geq h(b)$ with $a_1, a_2, b \in A$, $\frac{a_1 + a_2}{2} = b$, and $b$ not belonging to any subset $B \subset \conv(a_1, a_2) \cap A$ such that:
\begin{enumerate}
\item[(*)] for all $s \in B$ with $s \neq a_1, a_2$ there exist distinct $s_1, s_2 \in B$ satisfying $s_1 + s_2 = 2s$ and $\{s_1, s_2\} \neq \{a_1, a_2\}$.
\end{enumerate}
\end{proposition}

We give a simple example before the proof of the proposition.

\begin{example}
For $A = \{0,1,2,3,4\}$, the inequality $\frac{1}{2}(h(0)+h(4)) \geq h(2)$ is not a facet of $\mathcal M_A$, as it is implied by the other inequalities. The same is true for $A = \{0,2,3,4,6\}$ and the inequality $\frac{1}{2}(h(0)+h(6)) \geq h(3)$.
This is reflected by the fact that, in both of these cases, the set $A$ itself satisfies condition (*) above.  
However, for $A=\{0,3,4,5,8\}$, the inequality $\frac{1}{2}(h(0)+h(8)) \geq h(4)$ is a facet of $\mathcal M_A$.
\end{example}

\begin{proof}[Proof of Proposition \ref{prop:facetsmidpoint}]
Suppose the inequality $h(a_1) + h(a_2) \geq 2h(b)$ with $a_1, a_2, b \in A$, $a_1 + a_2 = 2b$ does not describe a facet of $\mathcal M_A$. 
This inequality must then be implied by other inequalities defining $\mathcal M_A$, that is, there exist $s_{i,1}, s_{i,2}, s_i \in A$ and $\lambda_i > 0$ such that $s_{i,1} + s_{i,2} = 2s_i$ for all $i$ and
\begin{equation}\label{eq:suminequalities}
\sum \lambda_i (e_{s_{i,1}} + e_{s_{i,2}} - 2e_{s_i}) = e_{a_1} + e_{a_2} - 2 e_{b} \quad \in \mathbb R^A,
\end{equation}
with $\{s_{i,1}, s_{i,2}\} \neq \{a_1, a_2\}$. 
Take $B$ to be the set containing all points $s_{i,1}, s_{i,2}, s_i$.
Note that $B \subset \conv(a_1, a_2) \cap A$, as Equation \eqref{eq:suminequalities} implies
that $\conv(B) = \conv(a_1, a_2)$.
Furthermore, for any $s \in B$ with $s \neq a_1, a_2$, since the coefficient of $e_s$ in Equation \eqref{eq:suminequalities} is nonpositive (equal to $0$ or $-2$), $s$ must be equal to some $s_i$,
and thus $s_{i,1}, s_{i,2} \in B$ satisfy $s_{i,1} + s_{i,2} = 2s_i$ and $\{s_{i,1}, s_{i,2}\} \neq \{a_1, a_2\}$. We thus see that $b$ belongs to a set $B$ that satisfies condition (*).

Suppose now that $a_1, a_2, b \in A$, $a_1 + a_2 = 2b$, and $b$ belongs to a subset $B \subset \conv(a_1, a_2) \cap A$ satisfying condition (*). 
Fix an indexing $s_1, \dots, s_m$ of the set $B$ in such a way that $B \setminus \{a_1, a_2\} = \{s_1,\dots, s_n\}$ and $s_1 = b$. 
For each $s_i$ choose distinct $s_{j_i}, s_{k_i} \in B$ satisfying $s_{j_i} + s_{k_i} = 2s_i$ and $\{s_{j_i}, s_{k_i}\} \neq \{a_1, a_2\}$.
Consider the matrix $M \in \mathbb R^{n \times n}$ whose diagonal entries $m_{i,i}$ are all equal to $-2$, entries $m_{j_i,i}$ and $m_{k_i,i}$ are equal to $1$ whenever $s_{j_i}$ and $s_{k_i}$ are not in $\{a_1, a_2\}$, and all other entries are equal to $0$.
Note that every column of $M$ has at most two entries equal to $1$, and all principal submatrices of $M$ must have a column with at most one entry equal to $1$; for instance, the column in the submatrix corresponding to the point $s_{i}$ that is closest to $a_1$.
By \cite[{Lemma 4.3}]{Reznick89} applied to the matrix $-M^{T}$, we see that the matrix $M$ is invertible and all the entries of $M^{-1}$ are nonpositive.
This implies that the linear system $M \cdot x = (-2, 0, \dots, 0)^{T}$ has a solution
$x \in \mathbb R^n$ with all entries being nonnegative.
This in turn means that we can write
\begin{equation}\label{eq:suminequalities2}
c_1 e_{a_1} + c_ 2 e_{a_2} - 2 e_{b} = \sum x_i (e_{s_{j_i}} + e_{s_{k_i}} - 2e_{s_i}) \quad \in \mathbb R^A,
\end{equation}
with all the $x_i$ being nonnegative.
Now, for any $u \in \mathbb R^n$, consider the vector $\hat u := \sum_{a\in A} (u \cdot a) e_a \in \mathbb R^A$. 
All the summands on the right hand side of Equation \eqref{eq:suminequalities2} are orthogonal to $\hat u$, and thus the same must be true for the left hand side, i.e., $u \cdot (c_1 a_1 + c_2 a_2 - 2b) = 0$. 
As this must be the case for any $u \in \mathbb R^n$, in fact we must have $c_1 a_1 + c_2 a_2 - 2b = 0$, and so $c_1 = c_2 = 1$.
Equation \eqref{eq:suminequalities2} thus shows that the inequality $h(a_1) + h(a_2) \geq 2h(b)$ is implied by other inequalities defining $\mathcal M_A$, and so it cannot define a facet of $\mathcal M_A$. 
\end{proof}

If $P \subset \RR^m$ is a convex cone, its dual cone is
\[P^{\vee} := \{ x \in \RR^m : x \cdot v \geq 0 \text{ for all } v \in P\}.\]
As we will see in Sections \ref{sec:tropmom} and \ref{sec:tsos}, the cones of convex and midpoint convex functions defined above are equal to the tropicalizations of the cones of $A$-moments and $A$-pseudo moments of measures supported on the entire nonnegative orthant $\R^n_{\geq 0}$. If we consider only measures supported on a semialgebraic subset $S$ of $\mathbb{R}_{\geq0}^n$ with the Hadamard property, this leads to a more general notion of convexity and midpoint convexity with respect to a polyhedral cone $C$. In this particular context, the cone $C$ is equal to $\trop (S)^\vee$, the dual cone to the tropicalization of $S$; when $S=\mathbb{R}_{\geq 0}^n$ we get $\trop (S)^\vee=\{0\}$.
Below we define these more general notions of convexity and study some of their combinatorial properties.

\begin{definition}\label{def:generalizedcones}
Suppose $A \subset \ZZ^n$ is finite and $C \subset \RR^n$ is a polyhedral cone. 
Define  
$$\textstyle \mathcal K_{A,C} := \{h : A \rightarrow \RR : \sum \lambda_i h(a_i) \geq h(b) \text{ whenever } b,a_i \in A, \lambda_i \geq 0, \sum \lambda_i = 1,\text{ and } \sum \lambda_i a_i - b \in C\}$$
and 
\begin{multline*}
\mathcal M_{A,C} := \{h : A \rightarrow \RR : h(a_1) + h(a_2) \geq 2 h(b) \text{ whenever } b,a_i \in A \text{ satisfy } a_1 + a_2 = 2b, \\ 
\text{and } h(a_1) \geq h(a_2) \text{ whenever } a_i \in A \text{ satisfy } a_1 - a_2 \in C\}.
\end{multline*}
We have $\mathcal K_{A,C} \subseteq \mathcal M_{A,C}$. 
Note that these cones generalize the cones of convex and midpoint-convex functions on $A$; 
indeed, we have $\mathcal K_A = \mathcal K_{A,\{0\}}$ and $\mathcal M_A = \mathcal M_{A,\{0\}}$. 
\end{definition}

It follows from the definition in~\eqref{eqn:tcone} that the tropical conical hull of  a subset $Y \subset \RR^A$ is 
\begin{multline*} 
\tcone(Y) := \{ h \in \RR^A : \exists h_1, \dots, h_r \in Y \text{ and } c_1, \dots, c_r \in \RR \text{ such that } \\ 
h(a) = \max (h_1(a) + c_1, \dots, h_r(a)+c_r) \text{ for all } a \in A\}.
\end{multline*}

In the following, we also use $A$ for the matrix whose rows are the elements of the set $A$ (in some order). 
\begin{theorem}\label{thm:K_{A,C}AsMax}
The cone $\mathcal K_{A,C}$ equals the tropical conical hull of $A (C^{\vee}) = \{A u\colon u\in C^\vee\}$. 
That is, a function $h:A \to \RR$ belongs to $\mathcal K_{A,C}$ if and only if there 
exist $u_1, \hdots, u_r\in C^{\vee}$ and $c_1, \hdots, c_r\in \R$ for which 
\begin{equation}\label{eq:maxLinear}
h(a) = \max\{ \langle u_1, a \rangle +c_1, \hdots, \langle u_r, a \rangle +c_r\}.
\end{equation}
\end{theorem}
\begin{proof}
Let $Y = A (C^{\vee})$. 
This is a polyhedral cone consisting of all functions of the form $h:A\to \R$ of the form $h(a)= \langle u, a\rangle$ for 
some $u$ in $C^{\vee}$. 
By definition, the tropical conical hull of $Y$ consists of functions of the form \eqref{eq:maxLinear}. 
First, we show that $\tcone(Y) \subseteq \mathcal{K}_{A,C}$.
Suppose $h\in \tcone(Y)$ and consider the extension $\hat{h}:\R^n \to \R $ obtained as the maximum of 
affine linear functions $\ell_j(x) = \langle u_j, x \rangle +c_j$.
Since $u_j\in C^{\vee}$, we have $\ell_j(x) \geq  \ell_j(y) $ whenever $x-y\in C$. 
It follows that $\hat{h}$ is convex and $\hat{h}(x) \geq  \hat{h}(y) $ whenever $x-y\in C$. 
In particular, if $\sum_i \lambda_i a_i - b \in C$, then 
\[
\sum_i \lambda_i  h(a_i) = \sum_i \lambda_i  \hat{h}(a_i)  \geq  \hat{h}\left(\sum_i \lambda_i a_i\right) \geq \hat{h}(b) = h(b).
\]
Therefore $\tcone(Y) \subseteq \mathcal{K}_{A,C}$.

For the reverse inclusion we show that the elements of $\mathcal{K}_{A,C}$ satisfy all the inequalities 
given by the dual cone $\tcone(Y)^\vee$. 
 The cone $Y$ is the image of $C^{\vee}$ under the linear map $A$. 
 By general convexity, the cone $Y^\vee$ is the intersection of $C$ with the image of the dual map of $A$. 
  Explicitly, for each $a\in A$ and $h\in \R^A$, define $\delta_a(h) = h(a)$.  
The dual map of $A$ maps a point $\sum_{a\in A}  \lambda_a\delta_a$ to $\sum_{a\in A}  \lambda_a a \in \R^n$. 

Suppose $\lambda = \sum_{a\in A}  \lambda_a\delta_a$ is an extreme ray of $\tcone(Y)^\vee$. 
By \Cref{prop:tconvdual}, this extreme ray only has one negative entry, say $\lambda_b$. We can rescale to make this entry $-1$. 
Then $\sum_{a\neq b} \lambda_a = 1$. This extreme ray 
imposes the condition $\sum_{a\neq b} \lambda_a h(a) \geq h(b)$ for elements $h:A\to \R$ of $\tcone(Y)$. 
Moreover, since $\lambda$ belongs to $\tcone(Y)^\vee$, its image $\sum_{a\in A}  \lambda_a a  = \sum_{a\neq b} \lambda_a a -b$
must belong to $C$.  By construction, $\lambda$ belongs to $\mathcal{K}_{A,C}^\vee$.  
\end{proof}

\begin{corollary}
The cone $\mathcal K_{A,C}$ consists of all functions $h:A \to \RR$ that can be extended to a 
convex function $\hat h : \RR^n \to \RR$ satisfying $\hat h(x_1) \geq \hat h(x_2)$ whenever $x_1 - x_2 \in C$. In particular, for any finite subset $E \subset \ZZ^n$ with $A \subset E$ we have
$$\mathcal K_{A,C} = \pi_{A}(\mathcal K_{E,C}) \subseteq \pi_{A}(\mathcal M_{E,C}).$$
\end{corollary}
\begin{proof} If $h \in \mathcal K_{A,C}$, one can check that the function 
$\hat{h}(x) = \max\{ \langle u_1, x \rangle +c_1, \hdots, \langle u_r, x \rangle +c_r\}$
given by \Cref{thm:K_{A,C}AsMax} is an extension of $h$ to $E$ that lies in $\mathcal M_{E,C}$. 
\end{proof}

The description by inequalities of $\mathcal K_{A,C}$ implies that the dual cone $\mathcal K_{A,C}^\vee$ is the conical hull of the vectors of the form $\sum \lambda_i e_{a_i} - e_b \in \RR^A$ with $b, a_i \in A$, $\lambda_i \geq 0$, $\sum \lambda_i=1$, and $\sum \lambda_i a_i - b \in C$. 
Similarly, the dual cone $\mathcal M_{A,C}^\vee$ is the conical hull of the vectors $e_{a_1} + e_{a_2} - 2 e_{b} \in \mathbb R^A$ with $b, a_i \in A$ and $a_1 + a_2 = 2b$ together with the
vectors $e_{a_1} - e_{a_2} \in \mathbb R^A$ with $a_i \in A$ and $a_1 - a_2 \in C$. 

\begin{proposition}\label{prop:face}
Suppose $F$ is a face of the polyhedral cone $C \subset \RR^n$. Then $\mathcal K_{A,F}^\vee$ is a face of $\mathcal K_{A,C}^\vee$, and similarly, $\mathcal M_{A,F}^\vee$ is a face of $\mathcal M_{A,C}^\vee$.
\end{proposition} 
\begin{proof}
Since $F$ is a face of $C$, there exists $u \in \mathbb R^n$ such that $u \cdot c = 0$ for all $c \in F$ and $u \cdot c > 0$ for all $c \in C \setminus F$. 
Consider the vector $\hat u := \sum_{a\in A} (u \cdot a) e_a \in \mathbb R^A$.
If $x = \sum \lambda_i e_{a_i} - e_b \in \RR^A$ is one of the generators of $\mathcal K_{A,C}^\vee$ described above, the dot product $\hat u \cdot x$ is equal to $\sum \lambda_i (u \cdot a_i) - (u \cdot b) = u \cdot (\sum \lambda_i a_i - b)$. This number is nonnegative for all the generators of $\mathcal K_{A,C}^\vee$, and it is equal to zero precisely when $\sum \lambda_i a_i - b \in F$. It follows that $\mathcal K_{A,F}^\vee$ is the face of $\mathcal K_{A,C}^\vee$ minimizing the dot product with $\hat u$. 

A similar argument shows that $\mathcal M_{A,F}^\vee$ is the face of $\mathcal M_{A,C}^\vee$ minimizing the dot product with $\hat u$. Indeed, if $x = e_{a_1} + e_{a_2} - 2 e_{b} \in \RR^A$ is one of the generators of $\mathcal M_{A,C}^\vee$ then $\hat u \cdot x = 0$, and if $x = e_{a_1} - e_{a_2} \in \mathbb R^A$ with $a_1 - a_2 \in C$ then $\hat u \cdot x = 0$ if and only if $a_1-a_2 \in F$.
\end{proof}

\begin{proposition}\label{prop:decreasingimplication}
Let $A \subset E \subset \ZZ^n$ be finite subsets and $C \subset \RR^n$ a pointed polyhedral cone. If $\mathcal K_{A,C} = \pi_{A}(\mathcal M_{E,C})$ then $\mathcal K_A = \pi_{A}(\mathcal M_E)$.
\end{proposition}
A combinatorial characterization of the sets of lattice points $A$ for which $\mathcal K_A = \pi_{A}(\mathcal M_E)$ with $E \supseteq \conv(A) \cap \mathbb Z^n$ was given in Corollary \ref{cor:projequality}.
\begin{proof}
Suppose $\mathcal K_{A,C} = \pi_{A}(\mathcal M_{E,C})$. Taking dual cones we obtain 
$$\mathcal K_{A,C}^\vee = \mathcal M_{E,C}^\vee \cap \mathbb R^A,$$
where $\mathbb R^A \subset \mathbb R^E$ denotes the coordinate subspace where $x_i =0$ for all $i \notin A$.
Since $C \subset \mathbb R^n$ is a pointed polyhedral cone, there exists $u \in \RR^n$ such that $u \cdot c > 0$ for all $c \in C \setminus \{0\}$. Let $\hat u := \sum_{a\in E} (u \cdot a) e_a \in \mathbb R^E$, and denote by $\pi(\hat u)$ its projection to $\RR^A$.
As explained in the proof of Proposition \ref{prop:face}, the face of $\mathcal K_{A,C}^\vee$ consisting of all vectors orthogonal to $\pi(\hat u)$ is equal to $\mathcal K_{A}^\vee$, and the face of $\mathcal M_{E,C}^\vee$ consisting of all vectors orthogonal to $\hat u$ is equal to $\mathcal M_{E}^\vee$.
Since a vector in $\mathbb R^A$ is orthogonal to $\pi(\hat u)$ if and only if it is orthogonal to $\hat u$, it follows that $\mathcal K_{A}^\vee = \mathcal M_{E}^\vee \cap \mathbb R^A$. 
Taking dual cones, we get $\mathcal K_{A} = \pi_{A}(\mathcal M_{E})$, as desired. 
\end{proof}

Proposition \ref{prop:decreasingimplication} shows that if the containment $\mathcal K_A \subset\pi_{A}(\mathcal M_E)$ is strict then for any fixed pointed cone $C \subset \R^n$ we also have $\mathcal K_{A,C} \subsetneq \pi_{A}(\mathcal M_{E,C})$.
Using the connection to moments and pseudo-moments described in Sections \ref{sec:tropmom} and \ref{sec:tsos}, this shows that even if we increase degree bounds for sums of squares, there will exist valid binomial inequalities in moments that are not satisfied by pseudo-moments.

\section{Tropicalizing the moment cone}\label{sec:tropmom}
\label{sec:moment}


Let $A \subset \ZZ^n_{\geq 0}$ be a finite subset. By a slight abuse of notation we will also denote by $A$ the matrix whose rows are the integer vectors of the point configuration $A$.   Let $\varphi_A: \mathbb{R}^n\rightarrow \mathbb{R}^{A}$, $x \mapsto (x^a \colon a\in A)$, be the corresponding monomial map.  

\begin{lemma}\label{lem:tropMA}
\label{lem:nonnegative}
If $S$ is a semialgebraic set in the nonnegative orthant $\RR_{\geq 0}^n$ with $S\subset \overline{S\cap \R^n_{>0}}$, then the tropicalization of its moment cone $M_A(S)$ is
\[
\trop(M_A(S)) = \tcone A (\trop(S)).
\]
\end{lemma}

\begin{proof}
Recall that the closed moment cone $M_A(S)$ can be described as
\[
M_A(S) = \overline{\cone(\varphi_A(S))}.
\] 
As seen in \Cref{subsec:TropDefs}, tropicalization only sees the points with nonzero coordinates, and taking closure does not affect tropicalization, so
\[
\trop M_A(S) = \trop \left(\cone(\varphi_A(S)) \cap \cR_{>0}^A\right).
\] 
Since $S\subset \overline{S\cap \R^n_{>0}}$ and $\varphi_A$ is a monomial map, we have $\varphi_A(S)\subset \overline{\varphi_A(S)\cap \R^A_{>0}}$ as well.  This implies that 
\[\cone\left(\varphi_A(S) \cap \cR_{>0}^A\right) = (\cone \varphi_A(S)) \cap \cR_{>0}^A = \cone\left(\varphi_A(S \cap \cR_{>0}^n)\right). \]
 By \eqref{eqn:TropConvCommuting} in \Cref{sec:tropConv},
 the tropicalization of the conical hull equals the tropical conical hull of the tropicalization for sets in the positive orthant, 
 so
\[
\trop M_A(S) = \tcone \trop \varphi_A(S\cap\cR_{>0}^n).
\]
On the other hand, tropicalization turns monomial maps into linear maps, so we have \[\trop(\varphi_A(S\cap \cR_{>0}^n)) = A \trop(S\cap \cR_{>0}^n) = A \trop(S) ,\]
 which proves the claim. 
\end{proof}

The above lemma allows us to give an elegant description of tropicalizations of subsets of $S\subset \R_{\geq 0}^n$ with the Hadamard property.

\begin{theorem}\label{thm:genmom}
  Let $S\subset \R_{\geq 0}^n$ be a semi-algebraic subset with the Hadamard property such that $S\subset \overline{S\cap \R^n_{>0}}$. The tropicalization of the moment cone $M_A(S)$ for a support set $A\subset \Z_{\geq 0}^n$ is the set of all functions $h\colon A \to \R$ satisfying the inequalities \begin{equation}\label{eqn:genmom}
  \sum_{i=1}^r \lambda_i h(a_i) \geq h(b)\,\,\,\, \text{for all} \,\,\,\, a_1,\ldots,a_r,b\in A, \,\, \lambda_i \geq 0, \sum_{i=1}^r \lambda_i = 1\,\,\, \text{with}\,\,\, \sum_{i=1}^r \lambda_i a_i - b \in \trop(S)^\vee.\end{equation}
  That is, $\trop(M_A(S)) = \mathcal K_{A,C}$ where $C = \trop(S)^\vee$. 
\end{theorem}

\begin{proof}
 By \Cref{lem:nonnegative}, $\trop(M_A(S)) = \tcone(A\trop(S))$. 
 Let $C = \trop(S)^\vee$. By \Cref{thm:K_{A,C}AsMax},  the cone $\mathcal{K}_{A,C}$ coincides with the tropical conical hull of $A (C^{\vee}) = A \trop(S)$.  
\end{proof}

The condition in \Cref{thm:genmom} cannot, in general, be split into two conditions, namely the convexity constraint ($\sum \lambda_i h(a_i) \geq h(b)$ for all convex combinations with $a_i,b\in A$) and the nonincreasing property ($h(a) \geq h(b)$ whenever $a-b\in \trop(S)^\vee$) as the following example shows.

\begin{example}\label{exm:nottwo}
 Consider $A = \{(3,0),(0,3),(2,2)\}$ and $S=[0,1]^2$. No point in $A$ is a convex combination of the others and for any $a\neq b\in A$, the difference $a-b$ does not lie in $\trop(S)^\vee = \R_{\leq 0}^2$. Therefore any function $h:A\to \R$ satisfies (1) and (2).  However, $\frac{1}{2}(3,0) + \frac{1}{2}(0,3) - (2,2)=(-\frac{1}{2}, -\frac{1}{2})$ belongs to $\trop(S)^\vee$ and so the functions $h:A\to \R$  in $M_A(S)$ satisfy $\frac{1}{2}h(3,0) + \frac{1}{2}h(0,3) \geq h(2,2)$. 
\end{example}

We now discuss implications of Theorem \ref{thm:genmom} for $\R_{\geq 0}^n$, the unit cube and toric cubes.

\subsection{Positive orthant}

Let $S = \RR_{\geq 0}$.  Then $\trop(S) = \RR^n$ and so $\trop(S)^\vee = \{0\}$. \Cref{thm:genmom} applied to this situation gives the following characterization of the tropicalization of the moment cone.
It recovers the result by Develin in~\cite{Develin_secant} that the tropical convex hull of a linear space $A(\trop(S))$ is the set of convex functions on $A$.

\begin{corollary}\label{prop:orthant}
The tropicalization  $\trop(M_A(\RR^n_{\geq 0}))$ is the cone of convex functions on $A$. \qed
\end{corollary}

\begin{exampleth}\label{exm_motz1}
Consider the Motzkin configuration in the plane shown in \Cref{fig:motzkintriangle}, i.e.~$A=\{(0,0),(1,2),(2,1),(1,1)\}$. By \Cref{prop:orthant} we see that
$\trop M_A(\RR^n_{\geq 0})$ consists of functions $h:A\rightarrow \mathbb{R}$ satisfying:
\begin{align}
  h(0,0)+h(1,2)+h(2,1) \geq 3 h(1,1).\tag*{$\diamond$}
\end{align}
\end{exampleth}

\subsection{The Cube} Let $C_n = [0,1]^n$ be the $n$-dimensional cube.  Then $\trop(C_n)$ is the nonpositive orthant of $\RR^n$. Since this orthant is also self-dual as a convex cone, \Cref{thm:genmom} implies the following description of the tropicalization of the moment cone on (hyper-)cubes.

\begin{corollary}\label{prop:CubeMoments}
The following four sets are equal to each other up to natural identifications.
\begin{enumerate}
\item $\trop(M_A(C_n))$, the tropicalization of the moment cone of the unit $n$-cube.
\item The tropical conical hull of the polyhedral cone spanned by columns
  of the matrix $-A$. 
\item The polyhedral cone in $\R^A = \{h\colon A \to \R\}$ defined by inequalities of the form
  \[\lambda_1 h(a_1) + \cdots + \lambda_r h(a_r) \geq h(b)\] 
  where $a_1,\ldots,a_r,b$ are in $A$,  $\lambda_i\geq 0$, $\sum \lambda_i=1$, and $\lambda_1 a_1 + \cdots + \lambda_r a_r$ is coordinate-wise at most $b$.
  \qed
\end{enumerate}
\end{corollary}

\begin{example}\label{exm:motzcube}
We revisit the Motzkin configuration in the plane of Example \ref{exm_motz1}. The tropicalization $\trop M_A(C_2)$ is the tropical conical hull of the two dimensional convex cone with extreme rays $(0,-1,-2,-1)$ and $(0,-2,-1,-1)$. Using Proposition  \ref{prop:tconvdual} we can compute that
$\trop M_A(C_2)$ consists of functions $h:A\rightarrow \mathbb{R}$ satisfying:
\[h(1,1)\geq h(1,2), \ \ h(1,1)\geq h(2,1), \ \ \text{and} \ \ h(0,0)+h(1,2)+h(2,1)\geq 3 h(1,1).\]
The corresponding binomial moment inequalities for measures supported on $[0,1]^2$ are:
\[m_{(1,1)}\geq m_{(1,2)}, \ \ m_{(1,1)}\geq m_{(2,1)},\ \ \text{and} \ \  m_{(0,0)}m_{(1,2)}m_{(2,1)}\geq m_{(1,1)}^3.\]
We only need to consider the pure binomials, by Remark~\ref{rmk:pureBinomials}.
\end{example}

The following example gives a better illustration of the difference between binomial moment inequalities for the orthant and the cube.
\begin{example}
Consider the following bivariate moments $A=\{(4,0),(0,4),(3,2)\}$. Since these points are in convex position, the tropicalization $\trop M_A(\mathbb{R}^2_{\geq 0})$ is all of $\RR^3$, and there are no binomial inequalities satisfied by these moments for measures supported on $\mathbb{R}_{\geq 0}$.

For the cube $C=[0,1]^2$ the tropicalization $\trop M_A(C_2)$ consists of functions $h:A\rightarrow \mathbb{R}$ satisfying:
\[h(4,0)+h(0,4)\geq 2h(3,2) \ \text{ and } \ 3h(4,0)+h(0,4)\geq 3 h(1,1).\]
These inequalities correspond to inequalities $\frac{1}{2}(4,0)+\frac{1}{2}(0,4) \leq (3,2)$ and $\frac{3}{4}(4,0)+\frac{1}{4}(0,4) \leq (3,2)$ coordinate-wise in $\mathbb{R}^2$. The corresponding binomial moment inequalities are:
\[m_{(4,0)}m_{(0,4)}\geq m_{(3,2)}^2 \ \text{ and } \ m_{(4,0)}^3m_{(0,4)}\geq m_{(1,1)}^3.\]
\end{example}

\subsection{Toric Cubes} Consider a subset $S$ of $[0,1]^n$ defined by binomial inequalities that is the closure of its positive points $S\cap (0,1)^n$. Such a set was called a \emph{toric cube} in \cite{toriccubes}. They also showed that every toric cube $S$ can be parametrized by a monomial map $\varphi:[0,1]^d\rightarrow [0,1]^n$ taking $\mathbf{t}=(t_1,\dots,t_d)$ to $(\mathbf{t}^{b_1},\dots, \mathbf{t}^{b_n})$ for appropriate exponent vectors $b_i$. Let $Q$ be a $d \times n$ matrix with columns $b_1, \dots, b_n$. We can think of $Q$ as a map from $\N^n$ to $\N^d$. 

Let $A=\{a_1,\dots, a_k\}\subset \N^n$ be a collection of lattice points and consider the moment cone $M_A(S)$. The pullback via the monomial parametrization $\varphi$ of $S$ gives a new configuration $A'$ whose points are given by $Qa_i$ for $1\leq i\leq k$.
So the moment cone $M_A(S)$ is the same as $M_{A'}([0,1]^d)$.
\begin{example}\label{exm_motztoric}
Let $S$ be the subset of $[0,1]^2$ given by the inequalities $y\geq x^3$ and $y\leq x^2$. We can see that $S$ is the image of $[0,1]^2$ under the map $(t_1,t_2) \mapsto(t_1t_2,t_1^2 t_2^3).$ Then $Q$ is given by $\begin{pmatrix} 1&2\\1&3\end{pmatrix}$.

Consider again the Motzkin configuration in the plane $A=\{(0,0),(1,2),(2,1),(1,1)\}$. The new configuration $A'=Q(A)$ is given by $A'=\{(0,0),(5,7),(4,5),(3,4)\}$. 
\end{example}

We apply \Cref{thm:genmom} to $M_{Q(A)}([0,1]^d)$ to obtain a description of $\trop \left(M_A(S)\right)$. We use that $\trop(S) = Q^T \trop([0,1]^d)$, which implies by duality, using $\trop([0,1]^d)^\vee = \R^d_{\leq 0}$, the description of the relevant cone appearing in \Cref{thm:genmom}, namely $\trop(S)^\vee = \{ h \in \R^n \colon Q h \in \R^d_{\leq 0} \}$.
\begin{corollary}\label{prop:ToricCube}
The following sets are equal. 
\begin{enumerate}
\item $\trop(M_A(S))$, the tropicalization of the moment cone
\item The tropical conical hull of the polyhedral cone spanned by columns
  of the matrix $-A'$, where the points in $A'=Q(A)$ are written as the rows.
\item The polyhedral cone of functions $h\colon A \to \R$ defined by inequalities of the form
  $\lambda_1 h(a_1) + \cdots + \lambda_r h(a_r) \geq h(b)$ where $\lambda_i \geq 0$, $\lambda_1+\cdots + \lambda_r =1$, and $\lambda_1 Q(a_1) + \cdots + \lambda_r Q(a_r) \leq b$ coordinate-wise in $\mathbb{R}^d$. \qed
\end{enumerate}
\end{corollary}

\begin{example}[Example \ref{exm_motztoric} continued]
 Using Proposition  \ref{prop:tconvdual} we can compute that
$\trop M_A(S)$ consists of functions $h:A\rightarrow \mathbb{R}$ satisfying:
\begin{align*}h(2,1)\geq h(1,2),\\
h(1,2)+2h(1,1)\geq 3 h(2,1),\\
h(0,0)+3h(2,1)\geq 4 h(1,1),\\
h(0,0)+h(1,2)+h(2,1)\geq 3 h(1,1).
\end{align*}
The corresponding binomial moment inequalities for measures supported on $S$ are:
\begin{align*}m_{(2,1)}\geq m_{(1,2)},\\
m_{(1,2)}m^2_{(1,1)}\geq m_{(2,1)}^3,\\
m_{(0,0)}m_{(2,1)}^3\geq m_{(1,1)}^4,\\
m_{(0,0)}m_{(1,2)}m_{(2,1)}\geq m_{(1,1)}^3.
\end{align*}
Observe that since $S$ is a subset of $[0,1]^2$ the binomial inequalities for $[0,1]^2$ from Example \ref{exm:motzcube} hold, but we also acquire additional inequalities.
\end{example}

\subsection{Measures supported on all of $\mathbb{R}^n$}
For measures supported on all of $\mathbb{R}^n$ we now find a nice description of the tropicalization of the moment cone, where even points in $A$ play a more prominent role.

The cone $\operatorname{trop} (M_{A}( \mathbb{R}^n))$ naturally contains $\operatorname{trop} (M_{A}( \mathbb{R}^n_{\geq 0}))$ because every measure with support in $\R^n_{\geq 0}$ is, by extension by $0$, a measure on $\R^n$.
So one way to formulate our question is to ask: which linear inequalities valid on $\operatorname{trop} (M_{A}( \mathbb{R}^n_{\geq 0}))$, generated by convexity inequalities on almost-empty simplices, are also valid on $\operatorname{trop} (M_{A}( \mathbb{R}^n))$?

We call an almost-empty lattice simplex \emph{even} if all of its vertices are even lattice points (whereas the interior lattice point is not required to be even). We now state the main theorem of this subsection; it is very similar to Proposition \ref{prop:orthant}, but we now only consider even almost-empty simplices.
\begin{theorem}\label{thm:whole}
A function $h:A\rightarrow \mathbb{R}$ belongs to $\operatorname{trop} (M_{A}( \mathbb{R}^n))$ if and only if $h$ is convex on all almost-empty even simplices in $A$.
\end{theorem}

We first establish helpful facts before proving this theorem.
\begin{lemma}\label{lem:evenmoment}
Let $A_+$ and $A_-$ be disjoint subsets of $A$. Suppose that a pure binomial inequality
$$
\prod_{\alpha \in A_+} |m_\alpha |^{a_\alpha}\geq \prod_{\beta \in A_-}  |m_\beta |^{a_\beta}
$$
is valid on $M_A(\mathbb{R}^n)$ for some $a_{\alpha}, a_{\beta}\in \Z_{> 0}$. 
Then all lattice points in $A_+$ must be even.
\end{lemma}

\begin{proof}
It suffices to show that for any $\gamma\in A$ with $\gamma\not\in 2\N$, 
there exists a measure on $\R^n$ with $m_{\gamma}=0$ and all other moments non-zero. 

We first deal with the univariate case $n=1$. 
We need to show that for any odd $k\in A$ there is a measure on $\R$ whose $k$th moment is $0$ and all others are non-zero.
We use Hamburger's moment problem (see  \cite[Section~3.1]{MR2383959}) and proceed by induction on $d = \lceil \max(A)/2 \rceil$. 
If the $(d+1)\times (d+1)$ matrix $(m_{i+j})_{0\leq i, j \leq d}$ is positive definite, 
then there is a measure $\mu$ on the real line with moments $(m_0, m_1, \hdots, m_{2d})$. 
 For $d=0$, there are no odd moments, and so there is nothing to show. 
Now let $m_k=0$ and, by induction, suppose that we have a choice of 
$(m_0, m_1, \hdots, m_{2d-2}) \in \R^{2d-1}$
so that  $(m_{i+j})_{0\leq i, j \leq d-1}$ is positive definite. 
If $k\neq 2d-1$, then let $m_{2d-1}$ be any nonzero real number. Otherwise, if $k= 2d-1$, we set $m_{2d-1} = 0$. 
For sufficiently large $m_{2d}\in \R_+$, the resulting matrix $(m_{i+j})_{0\leq i, j \leq d}$
will be positive definite, meaning that $(m_0, m_1, \hdots, m_{2d})$ is the vector of moments of 
some measure on $\R$, all of which are nonzero except for $m_k$. 
In fact, we can choose an atomic measure on $\R$, meaning that for some 
$t_1,\dots,t_{d+1}\in \R $ and nonnegative weights $w_1, \hdots, w_{d+1}\in \R_{\geq 0}$, 
$m_j = \sum_{i=1}^{d+1} w_i t_i^j$ for all $j = 0, \hdots, 2d$. 

Now we leverage the univariate case to prove the multivariate case. 
Suppose that for some $\gamma = (\gamma_1, \hdots, \gamma_n)\in A_+$, $\gamma_1$ is odd. 
Let $v = (v_1, \hdots v_n)$ be an element of $\N^n$ such that $v_1$ is odd and $v_2, \hdots, v_n$ are even, 
and $\langle \alpha, v\rangle  \neq \langle \beta, v\rangle$ for any pair of distinct elements $\alpha, \beta$ in $A$.
As $\langle \alpha, v\rangle  = \langle \beta, v\rangle$ defines a hyperplane in $\R^n$, there are many such choices. 

The image of $A$ under $\langle \cdot, v\rangle$ is a finite subset of $\N$. 
Moreover 
$\langle \gamma, v\rangle$ is odd. By the univariate case, there are some 
$t_1,\dots,t_{d}\in \R $ and nonnegative weights $w_1, \hdots, w_{d}\in \R_{\geq 0}$, 
so that $\sum_{i=1}^d w_i t_i^j$ is zero for $j = \langle \gamma, v\rangle$ 
and nonzero for all $j = \langle \alpha, v\rangle$ where $\alpha \in A\backslash\{\gamma\}$. 
We can extend this to the desired measure on $\R^n$ by considering an atomic measure 
with points  $t_i^{v} = (t_i^{v_1}, \hdots, t_i^{v_n})$ and the same weights. 
Then for any $\alpha\in \N^n$, the corresponding moment is 
\[
m_{\alpha} \ = \  \sum_{i=1}^d w_i \prod_{j=1}^n(t_i^{v_j})^{\alpha_j} \  =  \ \sum_{i=1}^d w_i t_i^{\langle v, \alpha \rangle}.
\]
By construction, this is a measure on $\R^n$ with $m_{\gamma}=0$ and $m_{\alpha}\neq 0$ for all other $\alpha$. 
\end{proof}

\begin{lemma}
\label{lem:mid}
 A function $h:A\rightarrow \mathbb{R}$ belongs to $\operatorname{trop} (M_A( \mathbb{R}^n))$ if and only if $h$ satisfies all inequalities $\sum_{a \in A} \lambda_a h(a) \geq 0$  valid on $\operatorname{trop} (M_A( \mathbb{R}_{\geq 0}^n))$ where $ \lambda_a > 0$ only for even points $a \in A$. 
\end{lemma}
\begin{proof}

We think in terms of pure binomial inequalities in absolute values of moments of measures. Suppose that a pure binomial inequality in absolute values of moments is valid on $M_A(\mathbb{R}^n)$. Then by Lemma \ref{lem:evenmoment} we know that this inequality only has even moments on the ``greater'' side. Moreover, this inequality must be valid for all measures supported on $\mathbb{R}^n$, and therefore also for all measures supported on $\mathbb{R}^n_{\geq 0}$. This gives us one inclusion.

Given a measure $\mu$ on $\mathbb{R}^n$ define a measure $|\mu|$ on $\mathbb{R}^n_{> 0}$ via $|\mu|(\mathcal{B})=\sum_{g \in \{\pm1\}^n} \mu(g\mathcal{B})$ for $\mathcal{B}\subseteq \R_{>0}^n$, where $g$ ranges over all possible coordinate sign changes.
We can also think of $|\mu|$ as a measure on all of $\mathbb{R}^n$ which gives zero weight to subsets outside of $\mathbb{R}^n_{>0}$. 

The $a$-th moment of $|\mu|$ is a sum of the absolute values 
of the $a$-th moment of the restriction of $\mu$ to each orthant. 
That is, $\int x^{a} \ d|\mu|$ equals $\sum_{g\in \{\pm1\}^n}|\int_{\R_g} x^{a} \ d\mu|$, where 
$\R_g$ denotes the orthant in $\R^n$ with sign pattern $g$. 
If $a$ is even, then each term is already positive and so the $a$-th moment of $\mu$ and $|\mu|$ agree. 
For arbitrary $a\in \N^n$, the $\alpha$-th moment of $\mu$ is upper bounded by that of $|\mu|$. 

Therefore if $\mu$ fails a pure binomial inequality in absolute values of moments, where all of the terms on the ``greater'' side are even, then $|\mu|$ will fail this inequality as well. This shows no additional pure binomial inequalities are valid on $M_A(\mathbb{R}^n)$ besides those that have all even terms on the ``greater'' side and are valid on $M_A(\mathbb{R}^n_{\geq 0})$.
\end{proof}

We can finish the proof of Theorem  \ref{thm:whole}.
\begin{proof}[Proof of Theorem \ref{thm:whole}]

Let $Y=\operatorname{trop} (M_{A}( \mathbb{R}^n))$ and $Z=\operatorname{trop} (M_{A}( \mathbb{R}^n_{\geq 0}))$. It follows from Lemma \ref{lem:mid} that $Y^\vee=Z^\vee\cap T$ where $T$ the set of vectors in $\mathbb{R}^{|A|}$ where coordinates with non-even indices are nonpositive. Since the cone $Y$ is tropically convex it follows from Proposition \ref{prop:tconvdual} that
$$Y^\vee= \sum_{\alpha \in A} (Y^\vee \cap U_\alpha),$$
where $\sum$ stands for Minkowski addition and $U_\alpha$ is the orthant of $\RR^{|A|}$ where the $\alpha$-indexed coordinate is nonpositive and the rest are nonnegative. Using the above centered equation on the $Y^\vee$ inside the Minkowski sum we see that
$$Y^\vee= \sum_{\alpha \in A} (Z^\vee \cap U_\alpha \cap T).$$
We observe that $U_\alpha\cap T$ is the set of vectors in $\RR^{|A|}$ where $\alpha$-indexed coordinate is nonpositive, all non-even indexed coordinates (except potentially $\alpha$) are zero, and all even-indexed coordinates (except potentially $\alpha$) are nonnegative. We see that  $U_\alpha\cap T$ is a face of $U_\alpha$, and therefore extreme rays of  $Z^\vee \cap U_\alpha \cap T$ are extreme rays of $Z^\vee \cap U_\alpha$. It follows from Lemma \ref{lem:emptysimplex} that the extreme rays of $Y^\vee$ come from almost-empty even simplices as desired.  \end{proof}

\begin{example}\label{exm:dm}
We consider the \emph{doubled Motzkin configuration} $\A=(0,0),(2,4),(4,2),(2,2)$ in the plane. By Theorem \ref{thm:whole} we see that
$\trop M_\A(\RR^n)$ consists of functions $h:\A\rightarrow \mathbb{R}$ satisfying:
$$h(0,0)+h(2,4)+h(4,2)\geq 3 h(2,2).$$
\end{example}

\subsection{Recovering The Moment Cone from Tropicalization} Although tropicalization is a quite forgetful operation in general, sometimes we can recover the moment cone from its tropicalization. If our configuration $A$ is an almost empty simplex then  the moment cone is defined by a single binomial inequality and thus completely determined by its tropicalization. See also \cite[Lemma~3.5]{DHNdW}

\begin{proposition}\label{prop:AMGM}
Let $A\subset \ZZ^n$ be an almost-empty simplex with vertices $v_0, \hdots, v_n$ and interior lattice point $w$. 
Write $w=\sum_{i\in I}\lambda_i v_i$ with $\lambda_i> 0$ for $i\in I$ and $\sum \lambda_i=1$. 
The moment cone $M_A(\R_{\geq 0}^n)$ equals the set of $(m_{v_0}, \hdots, m_{v_n}, m_w)\in \R_{\geq 0}^{n+2}$ 
for which $\prod_{i\in I}m_{v_i}^{\lambda_i}\geq m_w$. 
\end{proposition}

\begin{proof}
The proof proceeds by convex duality and we first consider the cone of polynomials supported on $A$ that 
are nonnegative on $\R_{\geq 0}^n$. The inequality 
\begin{equation}\label{eq:amgm}c_0x^{v_0}+\dots+c_{n}x^{v_n}\geq c_wx^{w}\end{equation} 
holds for all $x\in \R^n_{\geq 0}$ if and only if 
$c_0, \hdots, c_n$ are nonnegative and $c_w \leq \prod_{i=0}^n \left(\frac{c_i}{\lambda_i}\right)^{\lambda_i}$. 
This can be understood via the arithmetic-geometric mean inequality (in short AM/GM), which states that 
for every $y\in \R_{\geq 0}^{n+1}$, 
$$\lambda_0y_0+\dots+\lambda_{n}y_{n}\geq y_0^{\lambda_0}\cdots y_{n}^{\lambda_{n}},$$
with equality if and only if all coordinates $y_i$ with $i\in I$ are equal. 
Letting $y_i=(c_i/\lambda_i)x^{v_i}$ for $i\in I$ and applying AM/GM gives that 
$$c_0x^{v_0}+\dots+c_{n}x^{v_n}\geq 
\sum_{i\in I} c_i x^{v_i}
\geq 
x^w \prod_{i\in I} \left(\frac{c_i}{\lambda_i}\right)^{\lambda_i} .$$
The second inequality is tight when all the $y_i$'s are equal, meaning  $(c_i/\lambda_i)x^{v_i} = (c_j/\lambda_j)x^{v_j}$ for all $i,j$. 
We can find such a point $x\in \R_{>0}^n$ by solving the system of $|I|-1\leq n$ affine-linear equations 
$\langle \log(x), v_i- v_j\rangle = \log(c_j/\lambda_j)-\log(c_i/\lambda_i)$. 
To see that the first inequality must also be tight even when $I\neq \{0,\hdots, n\}$, 
consider a vector $\alpha\in \R^n$ so that 
$\langle \alpha,  v_i\rangle = a$ for all $i$ with $\lambda_i>0$ and  $\langle \alpha,  v_j\rangle < a$ whenever $\lambda_j = 0$. 
Consider rescaling $x\in \R^n$ coordinate-wise by $t^{\alpha}$. 
As $t\to \infty$, the limit of $t^{-a}(c_0(t^{\alpha}\cdot x)^{v_0}+\dots+c_{n}(t^{\alpha}\cdot x)^{v_n})$ 
equals $\sum_{i\in I} c_i x^{v_i}$. 
On the other hand, since $\sum_{i}\lambda_i v_i = w$, we see that  $\langle \alpha,  w\rangle = a$,  
$t^{-a}(t^{\alpha}\cdot x)^w$ equals $t^{\langle \alpha,  w\rangle -a}x^w = x^w$. 
So the right hand side in invariant under this rescaling. Since the inequality must hold for all $t$, 
we see that the coefficient of $x^w$ cannot be improved.

Using this characterization lets us derive defining inequalities for the dual cone $M_A(\R_{\geq 0}^n)$ as follows. 
The dual cone $M_A(\R_{\geq 0}^n)$ is contained in $\R_{\geq 0}^{n+2}$, and by convex duality, is 
 the set of points $(m_{v_0}, \hdots, m_{v_n}, m_w)$ for which 
$\sum_{i=0}^n c_i m_{v_i} \geq c_w m_w$ for all polynomials $\sum_{i=0}^n c_i x^{v_i} \geq c_w x^w$
nonnegative on $\R_{\geq 0}^n$. 
By coordinate scaling, we see that a point $(m_{\alpha})_{\alpha \in A}$ belongs to 
 $M_A(\R_{\geq 0}^n)$ if and only if $(m_{\alpha}x^{\alpha})_{\alpha \in A}$ belongs to 
 $M_A(\R_{\geq 0}^n)$ for every $x\in \R_{>0}^n$. 
 
 Let $c_0, \hdots, c_n\in \R_{\geq 0}$ and $c_w = \prod_{i\in I} \left(\frac{c_i}{\lambda_i}\right)^{\lambda_i} $. 
 By the arguments above, $\sum_{i} c_i m_{v_i}x^{v_i} \geq c_w m_w x^w$ for all $x\in \R_{>0}^n$ 
 if and only if 
\[c_w m_w \leq \prod_{i\in I} \left(\frac{c_i m_{v_i}}{\lambda_i}\right)^{\lambda_i} = c_w \prod_{i\in I} m_{v_i}^{\lambda_i}. 
\]
This shows that $M_A(\R_{\geq 0}^n)$ equals the set of $(m_{\alpha})_{\alpha \in A}$ in 
$\R_{\geq 0}^{n+2}$ satisfying $m_w \leq \prod_{i\in I}^n m_{v_i}^{\lambda_i}$. 
\end{proof}


\section{Truncated pseudo-moment cones}\label{sec:tsos}
\label{sec:SOS}

The goal of this section is to understand the tropicalizations of the dual cones of cones of sums of squares (e.g.~$\Sigma_A^\vee$) and more generally, dual cones to truncations of preorders and quadratic modules. We observe an interesting phenomenon of stabilization, which we now explain informally. 

Let $S$ be a compact basic closed semialgebraic set defined by inequalities $g_i(x) \geq 0$, $i=1,\dots,k$. We can build a set of obviously nonnegative polynomials on $S$ by combining sums of squares and the known nonnegativity of polynomials $g_i$; this set is called the \emph{preordering generated by $\{g_i\}_{i=1}^k$}:
 \begin{equation}
     \label{eqn:preorder}
     \left.
    \po(g_1,\ldots,g_k) = \left\{ \sum_{\alpha\in \{0,1\}^k} \sigma_\alpha g_1^{\alpha_1}g_2^{\alpha_2}\cdots g_k^{\alpha_k} \right\rvert \sigma_\alpha \text{ is a sum of squares of polynomials for each } \alpha \right\}.
 \end{equation}
By Schm\"{u}dgen's Positivstellensatz, any polynomial $f$ strictly positive on $S$ belongs to the preordering \cite{MR1092173}. This representation often involves sums of squares of degree significantly larger than that of the positive polynomial $f$.  We can then consider the ``truncated'' preordering where the degrees of the $\sigma_\alpha$ are bounded by some integer $d$. 
As $d \rightarrow \infty$, the truncated preorderings fill up the interior of the cone of nonnegative polynomials. We call the dual cones to truncated preorderings \emph{the cones of pseudo-moments}. They provide a convergent outer approximation to the cone of $A$-moments supported on $S$, as we allow degree bounds to grow.

To tropicalize the pseudomoment cones we restrict ourselves to sets $S$ contained in the nonnegative orthant and defined by binomial inequalities. Tropicalization of $M_A(S)$ encodes all pure binomial inequalities in $A$-moments of measures supported on $S$. One would expect that, as the degree bounds grow, tropicalizations of pseudo-moment cones provide a convergent approximation to $\trop M_A(S)$. However, as we will show below, often this does not happen. In fact tropicalizations of pseudo-moment cones \emph{stabilize}, so after a finite number of steps we do not learn any new binomial inequalities in moments from tropicalizing pseudomoments, even as the degrees grow arbitrarily large. We discuss stabilization when $S$ is the unit cube in \Cref{sec:cube} and for more general sets $S$ in \Cref{sec:stab}.

We ask in \Cref{ques:stabilization} whether this finite stabilization phenomenon always occurs for all sets in the nonnegative orthant defined by binomial inequalities.

We begin with the simple case of globally nonnegative polynomials, where one cannot pick up more obviously nonnegative polynomials by increasing the degree, since we have no generators $g_i$ to cancel out high-degree terms.

\begin{example}\label{tropicalizingpseudomoments}
 Consider the cone 
  \[\Sigma_{n,2k} = \left\{ \sum_{i=1}^r f_i^2 \colon r\in \N, \, f_i\in \R[x_1,x_2,\ldots,x_n]_{\leq k} \right\}\]
  in the real vector space $\R[x_1,x_2,\ldots,x_n]_{2k}$ with the monomial basis and tropicalize its dual cone $\Sigma_{n,2k}^\vee \subset \R[x_1,x_2,\ldots,x_n]^*$ in the dual basis. This cone $\Sigma_{n,2k}^\vee$ consists of all linear functionals $\ell\in\R[x_1,x_2,\ldots,x_n]^*$ such that the Hankel matrix $\left(\ell(x^{\alpha+\beta})\right)_{|\alpha|, |\beta|\leq k}$ is positive semidefinite, see for instance \cite[Section 4.6]{SOCAG}. 

  This Hankel matrix is a matrix filled with variables $y_\delta$, the dual coordinates in $\R[x_1,x_2,\ldots,x_n]^*$ corresponding to the monomial $x^{\delta}$. So we can apply \cite[Theorem 4.4]{BRST} to conclude that the tropicalization is given by the tropicalizations of the $2\times 2$ minors of the Hankel matrix.  The off-diagonal entries may be negative, but this does not affect the tropicalization which is defined by taking absolute values.
  These tropicalizations are linear inequalities of the form 
  \[y_{2\alpha} + y_{2\beta} \geq 2 y_{\alpha+\beta}.\]
  In other words, writing $\Delta_n = \{\alpha\in \Z_{\geq 0}^n \colon \sum_{i=1}^n \alpha_i \leq 1\}$, the set $\trop(\Sigma_{n,2k}^\vee)$ consists of all functions $h\colon k\,\Delta_{n} \to \R$ that satisfy \emph{even midpoint convexity}, i.e.~$h(v) + h(w) \geq 2 h(\frac12(v+w))$ for all $v,w \in k \Delta_{n}$ with only even coordinates.

  More generally, if we consider $\Sigma_{A} = \{ \sigma = \sum_{i=1}^r f_i^2 \colon r\in \N, \, \supp(\sigma) \subset A \}$ for a finite set $A\subset \Z^n$, then its dual cone $\Sigma_A^\vee$ is the projection of $\Sigma_{n,2k}^\vee$ (where $k$ is chosen such that $A\subset k\,\Delta_{n}$) onto $A$-coordinates because $\Sigma_A$ is a linear section of $\Sigma_{n,2k}$. Since tropicalization commutes with coordinate projections in this case, the same holds for the tropicalizations, i.e., the tropicalization of $\Sigma_A^\vee$ is the projection of the tropicalization of $\Sigma_{n,2k}^\vee$ onto $A$. 
  So a function $h\colon A \to \R$ is contained in $\trop(\Sigma_A^\vee)$ if and only if it can be extended to a function $\tilde{h}\colon \conv(A)\cap \Z^n \to \R$ such that $\tilde{h}(v) + \tilde{h}(w) \geq 2 \tilde{h}(\frac12(v+w))$ for all $v,w \in \conv(A) \cap 2\Z^n$.  We can further extend $\tilde{h}$ to a function on $k\,\Delta_{n}$ satisfying the same property, by assigning to points in $k\,\Delta_{n} \setminus \conv(A)$ values of a convex function which is sufficiently larger than the values of $\tilde{h}$ on $\conv(A) \cap \Z^n$.

  For the doubled Motzkin configuration $A=\{(0,0),(2,4),(4,2),(2,2)\}$ in the plane from \Cref{exm:dm}, there are no midpoints in $A$, and therefore $\trop \Sigma^\vee_{A}$ is $\RR^4$. This shows that tropicalization detects the difference between the cones $\Sigma_A^\vee$ and $M_A(\mathbb{R}^n)$.
\end{example}

More generally, we want to consider pseudo-moments on semi-algebraic subsets of the nonnegative orthant $\R_{\geq 0}^n$. So rather than considering the cone sums of squares, we consider truncated quadratic modules or preorders. 
These are convex cones inside the infinite-dimensional vector space $\R[x_1,x_2,\ldots,x_n]$ and we truncate them (that is to say look at finite-dimensional versions) in order to tropicalize them. There are various ways to do this and we want to illustrate the general procedure with the following basic case, the nonnegative orthant itself.

\begin{example}\label{tropicalizationpseudononneg}
  The nonnegative orthant $\R_{\geq 0}$ is defined by the inequalities $x_i \geq 0$ for $i=1,2,\ldots,n$ so we consider the preordering $\po(x_1,\ldots,x_n)$ defined as in \eqref{eqn:preorder}. To get a finite-dimensional version of this cone, we restrict it to the space $\R[x_1,x_2,\ldots,x_n]_A$ of all polynomials with support in $A\subset \Z^n$, where $A$ is a finite set. For simplicity, let us also assume that $A = \conv(A)\cap \Z^n$. 

  In this case, the preorder has a nice property, namely that the intersection of $\po(x_1,\ldots,x_n)$ with the finite-dimensional subspace $\R[x_1,x_2,\ldots,x_n]_A$ consists of exactly those elements of the form $f = \sum_{\alpha \in \{0,1\}^n} \sigma_\alpha x^\alpha$, where the sum of squares $\sigma_\alpha$ satisfy the obviously sufficient property that $(\NP(\sigma_\alpha) + \sum_{i=1}^n \alpha_i e_i)\cap \Z^n \subset A$ (see for example \cite[Theorem 5.2]{zbMATH05598928}). This is also necessary here because each extreme term of $\sigma_\alpha$ has a positive
  coefficient (since it is a sum of squares of real polynomials) and the extreme term of $f$ in any direction $v$ is attained at one of the extreme terms of $\sigma_\alpha x^\alpha$. 

  Write $PO_A$ for the intersection of $\po(x_1,\ldots,x_n)$ and $\R[x_1,\ldots,x_n]_A$. Setting $x_i = z_i^2$ for new variables $z_i$, we can identify $PO_A$ with the cone of sums of squares in $\R[z_1,z_2,\ldots,z_n]_{A'}$, where $A'$ consists of all points $2\alpha$ with $\alpha \in A$. By the previous \Cref{tropicalizingpseudomoments}, we conclude that the tropicalization of the dual cone $PO_A^\vee$ is the set of functions $h\colon A \to \R$ that are midpoint-convex, i.e.~satisfy $h(v) + h(w) \geq 2 h(\frac{1}{2}(v+w))$ for all $v,w\in A$ such that $\frac{1}{2}(v+w) \in A$.

  In other words $\trop(PO_A^\vee )$ consists of secondary cones of $A$ corresponding to subdivisions in which for every $v,w,x \in A$ satisfying $2x=v+w$, if $v$ and $x$ are marked in the subdivision, then $w$ is also marked.
\end{example}

The following result shows that tropicalization can detect, in some special situations at least, if nonnegative polynomials are sums of squares (in a dual sense).
\begin{proposition}
Let $A\subset \Z_{\geq 0}^n$ be an almost empty simplex with vertices $v_0,\dots,v_{n}$ and interior point $w$ with $w=\lambda_0v_0+ \lambda_1 v_1 + \ldots + \lambda_{n} v_{n}$ with $\lambda_i\geq 0$ and $\sum_{i=0}^n \lambda_i = 1$.
The cone of $A$-moments $M_A(\RR^n_{\geq 0})$ of measures supported on $\R_{\geq 0}^n$ is equal to $PO_A^\vee$ as in \Cref{tropicalizationpseudononneg} if and only if their tropicalizations agree, i.e.~$\trop \left(M_A(\RR^n_{\geq 0})\right)= \trop \left(PO_A^\vee\right)$.
\end{proposition}

\begin{proof}
The ``only if'' direction is clear, and we need to show that equality of tropicalizations implies equality of the actual cones. 

From \Cref{prop:AMGM}, the moment cone $M_A(\RR^n_{\geq 0})$ is defined by the inequality $\prod_{i\in I}m_{v_i}^{\lambda_i} \geq m_w$ where $I  = \{i: \lambda_i>0\}$. 
Its tropicalization  $\trop \left(M_A(\RR^n_{\geq 0})\right)$ is then given by the single  
linear inequality $\sum_{i\in I} \lambda_i y_{v_{i}} \geq y_w$. 
If $\trop \left(M_A(\RR^n_{\geq 0})\right)= \trop \left(PO_A^\vee\right)$, then this linear inequality also holds on  $\trop \left(PO_A^\vee\right)$. 
By Propositions~\ref{prop:binomIneq} and \ref{prop:TropConeLog}, it follows that $\prod_{i\in I}m_{v_i}^{\lambda_i} \geq m_w$ holds on $PO_A^\vee$. 
Therefore  $PO_A^\vee \subseteq M_A(\RR^n_{\geq 0})$. The other containment is immediate. 
\end{proof}

A difference in tropicalizations directly implies that the original cones are different, as in the following example.
\begin{example}
  We revisit the Motzkin configuration $A$ in the plane of \Cref{exm_motz1}. In this case there are no midpoints in $A$, and therefore $\trop PO^\vee_{A}$ (with the notation as in the previous example) is $\RR^4$. However, the tropicalization of the moment cone is a half-space. The difference between $\trop PO_A^\vee$ and $\trop M_A(\R_{\geq 0}^2)$ of course implies that $PO_A^\vee \neq M_A(\R_{\geq 0}^2)$ and dually, that there are nonnegative polynomials on $\R_{\geq 0}^2$ that are not in $PO_A$.
\end{example}

However, it can happen that the tropicalizations of truncated moment cones and truncated pseudo-moment cones are equal even though the original cones are different.
\begin{example}
  Let $A\subset \Z_{\geq 0}^n$ be the collection of nonnegative integer points with total sum of coordinates equal to $2$. The cone $P_A(\RR^n_{\geq 0})$ of quadratic forms in $n$ variables that are nonnegative on the nonnegative orthant is known as the cone of \textit{copositive quadratic forms}. The cone $PO_A$ (with the notation of \Cref{tropicalizationpseudononneg}) is the cone of quadratic forms that can be written as a sum of a positive semidefinite quadratic form and a form with nonnegative coefficients. It is known that for $n \geq 5$ we have strict containment $PO_A \subsetneq P_A(\RR^n_{\geq 0})$ (see \cite{copositive2}). Therefore we also have strict containment of the duals: $M_A(\RR^n_{\geq 0})  \subsetneq PO^\vee_{A}$. However, all convexity relations in $A$ are midpoint relations, so we have $\trop PO^\vee_{A}=\trop M_A(\RR^n_{\geq 0})$.
\end{example}

Truncation will often not have the nice properties that we see above on the nonnegative orthant, and we need to discuss truncation in general. Since our focus is on subsets of the nonnegative orthant, we restrict our attention to quadratic modules and preorders that are closed under multiplication by monomials. 
Let $g_1,g_2,\ldots,g_r$ be fixed polynomials in $\R[x_1,x_2,\ldots,x_n]$ (that we usually assume to be binomials or even pure binomials). We fix the following notation: Set $g_0 = 1$ and let $\Sigma^2$ denote the cone of sums of squares 
in $\R[x_1, \hdots, x_n]$. We define the truncated quadratic module and preorder generated by $g_1, \hdots, g_r$ to be 
\begin{align*}
  {\rm QM}_d(g_1,\ldots,g_r) & = \left\{ \sum_{J\subseteq[n], i\in \{0\}\cup [r]} \sigma_{J,i} x^Jg_i : \sigma_{J,i} \in \Sigma^2, \deg(\sigma_{J,i} x^Jg_i)\leq d \text{ for all }J, i \right\}, \text{ and }  \\
  {\rm PO}_d(g_1,\ldots,g_r) & = \left\{ \sum_{J\subseteq[n], I\subseteq [r]} \sigma_{J,I} x^Jg^I : \sigma_{J,I} \in \Sigma^2, \deg(\sigma_{J,I} x^Jg^I)\leq d \text{ for all }J, I\right\},
\end{align*}
respectively, where $g^I$ denotes $\prod_{i\in I}g_i$ with the convention that $g^{\emptyset} = 1$. Moreover, we slightly abuse notation and identify a subset $J$ with its indicator vector in $\{0,1\}^n$. 

The dual cones are defined in terms of (pseudo-)moments and {\bf localizing matrices} as follows.
For degree $d\in \N$, linear function $\ell \in \R[x_1,x_2,\ldots,x_n]_d^*$, and polynomial 
$f\in \R[x_1, \hdots, x_n]$ 
with degree $\leq d-k$, 
we define the symmetric matrix 
\[\M{k}{f}[\ell] = (\ell(x^{\gamma + \delta}f))_{|\gamma|, |\delta| \leq k}\]
 whose rows and columns are indexed by the monomials $x^\gamma, x^\delta$ of degree at most $k$, ordered lexicographically.    
 To simplify notation, here we write $\M{}{J,i}[\ell] = \M{d(J,i)}{x^J g_i}[\ell]$ and $\M{}{J,I}[\ell] = \M{d(J,I)}{x^J g^I}[\ell]$, where $d(J,i)=\lfloor \frac{d - \deg(x^Jg_i)}{2} \rfloor $ and $d(J,I)=\lfloor \frac{d - \deg(x^Jg^I)}{2} \rfloor $, when the choice of $d$ is clear from context.
 
Then the dual cones of the truncated quadratic module and preorder above are given by positivity conditions on the localizing matrices:
\begin{align*}
  {\rm QM}_d(g_1,g_2,\ldots,g_r)^\vee &= \left\{ \ell \in \R[x_1,x_2,\ldots,x_n]_d^* \colon 
  \M{}{J,i}[\ell] \succeq 0 \text{ for all }J, i \right \}\\
  {\rm PO}_d(g_1,g_2,\ldots,g_r)^\vee &= \left\{ \ell \in\R[x_1,x_2,\ldots,x_n]_d^* \colon \M{}{J,I}[\ell]\succeq 0 \text{ for all }J, I\right\}.
\end{align*}

\subsection{The pure binomial case}\label{subsec:pseudoPureBinomial}
We work with semi-algebraic subsets $S\subset \R_{\geq 0}^n$ of the nonnegative orthant defined by pure binomial inequalities $x^a - x^b \geq 0$. 
For pure binomials, we have \[\M{k}{x^a - x^b}[\ell] = \M{k}{x^{a}}[\ell] - \M{k}{x^{b}}[\ell]\] This is a matrix of the form $A - B$ as in \Cref{prop:binom}, where the entries of $A = \M{k}{x^{a}}[\ell]$ and 
$B = \M{k}{x^{b}}[\ell]$ are variables corresponding to the coordinates $(\ell(x^{\gamma}))_{|\gamma|<d}$ of the linear functional~$\ell$. We use this to show that the cone ${\rm QM}_k(g_1,g_2,\ldots,g_r)^\vee$ has the Hadamard property
when the polynomials $g_i$ are pure binomials. 

\begin{lemma}\label{lem:QM_Hadamard}
Let $g_1, \hdots, g_r$ be pure binomials and consider the semi-algebraic set $S\subset\R^n$ defined by $x_i \geq 0$ and $g_j\geq 0$. 
If $S$ is full-dimensional, then ${\rm QM}_d(g_1, \hdots, g_r)^\vee$ has non-empty interior and the Hadamard property. 
\end{lemma}

\begin{proof}
If $S$ is full-dimensional, then there can be no polynomial that is both nonnegative and nonpositive on $S$. 
That is, the convex cone of polynomials in $\R[x_1, \hdots, x_n]_d$ that are nonnegative on $S$ is pointed. 
Since ${\rm QM}_d(g_1, \hdots, g_r)$ is a subset of this cone, it is also pointed, implying that its dual cone is full dimensional. 

To see that ${\rm QM}_d(g_1, \hdots, g_r)^\vee$ has the Hadamard property, 
we first consider the case of a single binomial ($r=1$), say $g = x^{a} - x^{b}$. 
The cone ${\rm QM}_d(g)^\vee$ is defined by the positive semidefiniteness of matrices 
of the form $\M{}{J,\emptyset}\succeq 0$ and $\M{}{J,1}\succeq 0$.

For the second type of inequalities, fix $J\subseteq [n]$ and $k=d(J,1)$. We can write \[\M{k}{x^Jg}= \M{k}{x^{J+a}} - \M{k}{x^{J+b}}.\]
The two matrices on the right hand side are principal submatrices of the localizing matrices of squarefree monomials. 
To be precise write $J + a = K+ 2\alpha'$ and $J + b = L + 2\beta'$ where $K,L\subseteq [n]$ and $\alpha', \beta'\in \Z_{\geq 0}^n$. 
Then $\M{k}{x^{J+a}} = \M{k}{x^{K+2\alpha'}}$ is a principal submatrix of  $\M{m}{x^{K}}$
for $m \geq k +|\alpha'|$. Similarly $\M{k}{x^{J+b}} = \M{k}{x^{L+2\beta'}}$ is a principal submatrix of  $\M{m}{x^{L}}$ for $m \geq k +|\beta'|$.   

Consider the spectrahedron $\mathcal{S}_J$ in $\R[x_1, \hdots, x_n]_d^*$ 
defined by the conditions $\M{}{K,\emptyset}\succeq 0$, $\M{}{L,\emptyset}\succeq 0$, and 
$\M{}{J,1} \succeq 0$. These constraints have the form $A\succeq 0$, $B\succeq 0$, and $A_{I} - B_{I'}\succeq 0$, 
where $A_{I}$, $B_{I'}$ are principal submatrices of $A$ and $B$, respectively. 
By \Cref{prop:binom}(1), the spectrahedron $\mathcal{S}_J$ has the Hadamard property. 
Since ${\rm QM}_d(g)^\vee$ is the intersection of $\mathcal{S}_J$ over all $J\subseteq [n]$, we see that 
 ${\rm QM}_d(g)^\vee$ has the Hadamard property as well. 

Similarly, for $r>1$, we can write ${\rm QM}_d(g_1, \hdots, g_r)^\vee$ as the intersection of ${\rm QM}_d(g_i)^{\vee}$ 
for $i=1, \hdots, r$. Since each individual set has the Hadamard property, so does their intersection. 

All together, we have that ${\rm QM}_d(g_1, \hdots, g_r)^\vee$ is a full-dimensional set in $\R[x_1, \hdots, x_n]_d^*$ 
with the Hadamard property.
\end{proof}

We now determine tropicalization of truncated quadratic modules for the case where all polynomials $g_i$ are pure binomials. 

\begin{theorem}\label{thm:genpseudomom}
Let $g_1, \hdots, g_r$ be pure binomials and consider the semi-algebraic set $S\subset\R^n$ defined by $x_i \geq 0$ and $g_j\geq 0$. 
  Assume $S$ is full dimensional, $S \subset \overline{S\cap \R_{>0}^n}$, and that the vectors $w_i = a_i - b_i$, where $g_i = x^{a_i}-x^{b_i}$, generate the semigroup $N = \trop(S)^\vee \cap \Z^n$. 
  Then, for any integer $d \geq 0$ the tropicalization of $\qm_d(g_1,\dots,g_r)^\vee$ is the rational polyhedral cone given by the following inequalities: 
  \begin{enumerate}
	  \item $h(u_1) + h(u_2) \geq 2 h(v)$ for all $u_1,u_2,v$ such that $|u_i| \leq d$, $|v|\leq d$ and $u_1 + u_2 = 2v$;
	  \item $h(u) \geq h(v)$ whenever $|u|\leq d$, $|v|\leq d$, and $u-v\in \trop(S)^\vee$. 
  \end{enumerate}
\end{theorem}
 

\begin{proof}
As seen above, the quadratic module $\qm_d(g_1,\dots,g_r)^\vee$ is cut out by  inequalities of the form  \[\M{}{J,\emptyset}\succeq 0,~ \text{ and }\M{}{J,i} \succeq 0\] for all $i=1,\dots,r$, $J\in \{0,1\}^n$.
By \Cref{lem:QM_Hadamard}, the toric spectrahedron $\qm_d(g_1,\dots,g_r)^\vee$ has non-empty interior and is full-dimensional, so we can apply \Cref{prop:binom}, which gives two types of linear inequalities defining the tropicalization  $\trop\qm_d(g_1,\dots,g_r)^\vee$. 

The first inequality type comes from the tropicalization of the $2\times 2$ minors of the localizing matrices $\M{}{J,\emptyset}$ of monomials.  For $J \in \{0,1\}^n$, the tropicalization of the $2\times 2$ minor indexed by the exponent vectors $\alpha_1$ and $\alpha_2$ gives the inequality
  \[
  h(2\alpha_1+J)+h(2\alpha_2+J)\geq 2 h\left(\frac{\alpha_1+\alpha_2}{2}+J\right).
  \]
  As we run through all combinations of $\alpha_1, \alpha_2$, and $J$, we get inequalities of type (1).  The $2\times 2$ minors of $\M{d(J,i)}{x^{a_i+J}}$ and $\M{d(J,i)}{x^{b_i+J}}$ give a subset of these inequalities, so we do not need to consider them separately.
  
The second inequality type comes from tropicalizing the diagonal entries of the localizing matrices \[\M{}{J,i}=\M{d(J,i)}{x^{a_i+J}} - \M{d(J,i)}{x^{b_i+J}}.\]  We claim that they are  the inequalities $h(u) \geq h(v)$ for $u-v\in \trop(S)^\vee \cap \Z_{\geq 0}^n$.  
Since the vectors $w_i = a_i - b_i$ generate the semigroup $\trop(S)^\vee \cap \Z_{\geq 0}^n$ by assumption, we can write  $u-v = \sum_j \lambda_j w_j$ where each $\lambda_j\in \ZZ_{\geq 0}$. Without loss of generality, assume that $\lambda_1 > 0$ so that $u - v = w_1 + \sum_j \lambda_j' w_j$ for $\lambda_j' \in \Z_{\geq 0}^n$. Write the vector $\sum_j \lambda_j' w_j$ as $2\alpha + J$ with $\alpha \in \Z_{\geq 0}^n$ and $J\in \{0,1\}^n$. The desired inequality is the tropicalization of the diagonal entry of $\M{}{J,1}$ indexed by $\alpha$ where $g_1 = x^{a_1} - x^{b_1}$ is the binomial corresponding to the exponent vector $w_1 = a_1 - b_1$. 
\end{proof}

The assumption that the vectors $v_i = a_i - b_i$ coming from the defining binomials of $S$ generate the semigroup $\trop(S)^\vee \cap \Z^n$ is not a strong restriction in the sense that we can add redundant binomial inequalities to a chosen inequality description of $S$ so that the assumption is satisfied, as we now show below.
\begin{proposition}
  Let $S\subset \R_{\geq 0}$ be a semi-algebraic set defined by pure binomial inequalities $g_i = x^{a_i} - x^{b_i}$ with the property that $S \subset \overline{S\cap \R_{>0}^n}$.
  By adding valid binomial inequalities to the description of $S$ as a semi-algebraic set, we can assume that the exponent vectors $a_i - b_i$ of the binomials generate the semigroup $\trop(S)^\vee \cap \Z^n$.
\end{proposition}

\begin{proof}
  By \Cref{lem:hadamardForS}, the set $S$ has the Hadamard property so that $\trop(S)$ is a convex cone. Since the polynomials $g_i = x^{a_i} - x^{b_i}$ are pure binomials, the vector $v_i = a_i - b_i$ associated to the exponents is in $\trop(S)^\vee$, which shows $\cone(v_1,\ldots,v_r) \subset \trop(S)^\vee$. For the reverse inclusion, suppose that $\trop(S)^\vee$ is strictly larger. Dually, this means that $\cone(v_1,\ldots,v_r)^\vee$ is strictly larger than $\trop(S)$. Since the linear inequality $\langle v_i,y\rangle \geq 0$ exponentiates exactly to $x^{a_i} - x^{b_i}\geq 0$ for $x_i = \exp(y_i)$, this shows that $\cone(v_1,\ldots,v_r)^\vee = \trop(S)$, a contradiction (see \Cref{prop:binomIneq}). 
  
  This implies that every element of a generating set for the semigroup $\trop(S)^\vee \cap \Z^n$ is a conic combination of $v_1,\ldots,v_r$. By exponentiating, such inequalities are redundant for $S$ and we can therefore add them all to the inequality description.
\end{proof}

Now fix a finite set $A \subset \Z_{\geq 0}^n$. Write $\R[x_1,x_2,\ldots,x_n]_A$ for the vector space of polynomials whose support is contained in $A$. By intersecting convex cones of polynomials with $\R[x_1,x_2,\ldots,x_n]_A$, we see that, up to changing degree bounds, tropicalizations of dual cones to quadratic modules and preorders behave in the same way. 
\begin{lemma}\label{lem:qmvspo}
  Let $S\subset \R^n$ be a full-dimensional semi-algebraic set defined by the inequalities $x_i\geq 0$ for $i\geq 0$ and pure binomial inequalities $g_i = x^{a_i} - x^{b_i}\geq 0$ ($i=1,2,\ldots,r$) with $a_i,b_i \in \ZZ_{\geq 0}^n$.
  Let $A\subset \Z_{\geq 0}^n$ be fixed and for $d\geq \max\{|\gamma|:\gamma\in A\}$, denote by $\pi_{A}$ the projection  $\pi_A :\R[x_1,x_2,\ldots,x_n]_d^* \to \R[x_1,x_2,\ldots,x_n]_A^*$ given by restriction. Consider the cones
  \[
    Q_d = {\rm QM}_d(g_1, \hdots, g_r) \text{ and  } P_d = {\rm PO}_d(g_1,\ldots,g_r).
  \]
  \begin{enumerate}
    \item The tropicalization of $P_d^\vee$ is contained in the tropicalization of $Q_d^\vee$ after projection, i.e.
    \[ \pi_A(\trop(Q_d^\vee) )\supseteq \pi_A(\trop(P_d^\vee)). \] 
    \item For sufficiently large $D\in \N$, the tropicalization of $Q_D^\vee$ is contained in the tropicalization of $P_d^\vee$ after projection, i.e.~$\pi_A(\trop(Q_D^\vee))\subseteq \pi_A(\trop(P_d^\vee))$.
  \end{enumerate}
\end{lemma}

\begin{proof}
  Part (1) is immediate: Since $Q_d \subseteq P_d$ and convex duality reverses inclusion, we have $Q_d^\vee \supseteq P_d^\vee$. Taking tropicalization preserves these inclusions, which implies the claim. 

  For Part(2), let $D \geq n+2d+2e$ where $e = \deg(\prod_{i\in [r]}g_i)$. By \Cref{lem:QM_Hadamard}, $\trop(Q_D^\vee)$ is a full dimensional convex cone. Therefore it suffices to show that $\pi_A(y)$ is in $\pi_{A}(\trop(P_d^\vee))$ for every interior point $y\in \trop(Q_D^\vee)$.
  
  Consider a point $y= (y(\gamma))_{|\gamma|\leq D}$ in the interior of $\trop(Q_D^\vee)$. Then the tropicalization of the diagonal and principal $2\times 2$ minors of the localizing matrices $\M{k}{x^Jg_i}$ are strictly positive at $y$ for $2k + \deg(x^Jg_i)\leq D$. In particular, 
  \[
    y(a_i+J+ 2\gamma) > y(b_i+J+ 2\gamma)
    \text{ and } y(J + 2\gamma) +y(J + 2\delta) >  2y(J + \gamma+ \delta) 
  \]
  for all $J \subseteq [n]$, $i=1, \hdots, r$, and $|\gamma|, |\delta|\leq k$. We can write any $\delta\in \Z_{\geq 0}^n$ in the form $\delta = J + 2\gamma$ for some $J\subseteq [n]$ and $\gamma\in \Z_{\geq 0}^n$, giving that $y(a_i+\delta) > y(b_i+\delta)$ for all $|\delta|\leq D -\deg(g_i)$, and so in particular all  $|\delta|\leq D -e$. 

  Let $\ell \in \R[x_1,x_2,\ldots,x_n]^*_d$ be defined by $\ell(x^{\gamma}) = t^{y({\gamma})}$, where we will choose $t\in \R_{>0}$ sufficiently large. Now consider the localizing matrix  $\M{m}{x^Jg^I}[\ell]$ for some $I\subseteq [r]$ where $2m +\deg(x^Jg^I)\leq d$. By linearity, we can expand $g^I$ and write
  \[
    \M{m}{x^Jg^I}[\ell] = \M{m}{x^{J + \sum_{i\in I}a_i}}[\ell] 
    + \sum_{\emptyset \neq T\subseteq I }(-1)^{|T|}\M{m}{x^{J + \sum_{i\in I\backslash T}a_i  +  \sum_{j\in T}b_j}}[\ell].
  \] 
  We claim that this is a polynomial matrix of the form required for \Cref{lem:manyBinom}.  The entries in each of the above matrices are variables of the form $\ell(x^{\gamma})$ for $|\gamma|\leq m$. First, we check that $T=\emptyset$ maximizes the tropicalization of the diagonal entries in the above sum. For all $T\neq \emptyset$ and all $|\gamma|\leq m$, we have 
  \[
    y\left(J+ 2\gamma+ \sum_{i\in I}a_i\right) > y\left(J+ 2\gamma+ \sum_{i\in I\backslash T}a_i + \sum_{j\in T}b_j\right)
  \]
  This follows from the inequalities $y(a_i+\delta) > y(b_i+\delta)$ for all $|\delta|\leq n+2m +e \leq D-e$ and induction on $|T|$.  

  Now fix some $T\subseteq I$ and let $\eta$ denote the integer vector $J + \sum_{i\in I\backslash T}a_i  +  \sum_{j\in T}b_j$. If $J'\subseteq[n]$ is the set of indices for which $\eta$ is odd, then we can rewrite $\eta = J' + 2\theta$  for some $\theta\in \Z_{\geq 0}^n$. 
  Consider the principal $2\times 2$ minor of $\M{m}{x^{\eta}}[\ell]$ given by $\gamma\neq \delta$ with $|\gamma|, |\delta|\leq m$. To check that the tropicalization of this minor is positive at $y$, we need to show that 
  \[ 
    2y (\eta +\gamma + \delta)< y(\eta+ 2\gamma) + y(\eta + 2\delta).
  \]
  Indeed, by using the $2\times 2$ minor of $\M{d}{x^{J'}}[\ell]$ corresponding to $(\theta+\gamma, \theta+\delta)$, we see that 
  \begin{align*}
    2y (\eta +\gamma + \delta)  = 2y (J' + 2\theta +\gamma + \delta) & < y (J' + 2(\theta+\gamma)) + y (J' + 2(\theta+\delta))\\ 
    & = y(\eta+ 2\gamma) + y(\eta + 2\delta).
  \end{align*}

  By \Cref{lem:manyBinom}, the matrix $\M{m}{x^Jg^I}[\ell]$ is positive definite for sufficiently large $t$. We can take $t$ large enough so that this holds for all $J\subseteq [n]$ and $I\subseteq [r]$, giving $\ell \in P_d^{\vee}$ and $y\in \trop(P_d^{\vee})$. 
\end{proof}

\subsection{The cube}\label{sec:cube}
As an application, we consider the $0/1$-cube $[0,1]^n\subset \R^n$ with its natural inequality description $x_i\geq 0$ and $g_i = 1-x_i\geq 0$ for $i=1,2,\ldots,n$. We apply the results of the previous section to show not only that the tropicalizations of the dual cones of the truncated preorder and quadratic module agree in sufficiently high degree but we give a nice combinatorial description of when this happens.

\begin{lemma}\label{lem:extensioncubical}
  Let $a_1,a_2,\ldots,a_n\in \N$ and let $B = \left([0,a_1]\times [0,a_2] \times \ldots [0,a_n]\right)\cap \Z^n\subset \R^n$ be a lattice box.
  Let $h\colon B\to \R$ be a midpoint convex function that is non-increasing in coordinate directions (i.e.~$h(\alpha + e_i)\leq h(\alpha)$ for all $\alpha\in B$ such that $\alpha + e_i \in B$). 
  Then the extension $\widehat{h}\colon \Z_{\geq 0}^n\to \R$ that assigns to every $\alpha\in \Z_{\geq 0}^n$ the value $h(\gamma)$, where $\gamma$ is the lattice point in $B$ that is closest to $\alpha$, is midpoint convex and non-increasing in coordinate directions.
\end{lemma}

\begin{proof}
  To prove this claim, we define the map 
  \[ \varphi\colon \left\{ 
    \begin{array}{l}
    \Z_{\geq 0}^n \to B, \\
    \varphi(\alpha) = (\min\{\alpha_1,a_1\},\min\{\alpha_2,a_2\},\ldots,\min\{\alpha_n,a_n\}).
    \end{array}\right.
  \]
  Then the extension $\wh{h}$ is the equal to $h\circ \varphi$. It is straightforward to see that $\wh{h}(\alpha + e_i) \leq \wh{h}(\alpha)$ because $\varphi(\alpha+e_i)$ is either equal to $\alpha$ or $\alpha + e_i$.

  So we only have to show that $\wh{h}(\alpha) + \wh{h}(\beta) \geq 2 \wh{h}(\delta)$ for all $\alpha,\beta,\delta\in \Z_{\geq 0}^n$ with $\alpha + \beta = 2 \delta$.

  We prove this by case distinction and analysis of $\varphi$. 
  For this, we construct two vectors $v,w\in\Z^n_{\geq 0}$ such that $\varphi(\alpha) + v, \varphi(\beta) + w \in B$ and $\varphi(\alpha) + v + \varphi(\beta) + w = 2\varphi(\delta)$: 
  If $\alpha_i< a_i$ and $\beta_i < a_i$, set $v_i = 0$ and $w_i = 0$. If $\alpha_i \geq a_i$ and $\beta_i \geq a_i$, again set $v_i = 0$ and $w_i = 0$. If $\alpha_i < a_i$, $\beta_i \geq a_i$, and $\delta_i \geq a_i$, set $v_i = a_i - \alpha_i$ and $w_i = 0$. If $\alpha_i < a_i$, $\beta_i \geq a_i$, and $\delta_i <a_i$, set $v_i = \beta_i - a_i$ and $w_i = 0$. Symmetrically, if $\alpha_i \geq a_i$, $\beta_i < a_i$, and $\delta_i \geq a_i$, set $v_i = 0$ and $w_i = a_i - \beta_i$. Lastly, if $\alpha_i \geq a_i$, $\beta_i < a_i$, and $\delta_i \geq a_i$, set $v_i = 0$ and $w_i = \alpha_i - a_i$. 
  
  With this, we are done, because 
  \[
    \begin{array}{l}
      \wh{h}(\alpha) + \wh{h}(\beta) = h(\varphi(\alpha)) + h(\varphi(\beta))\geq \\
      \geq h(\varphi(\alpha) + v) + h(\varphi(\beta) + w) \geq 2 h(\varphi(\delta)) = 2 \wh{h}(\delta),
    \end{array}
  \]
  where we used the midpoint convexity of $h$ since $\varphi(\alpha) + v + \varphi(\beta) + w = 2\varphi(\delta)$.
\end{proof}

Fix a finite set $A\subset \Z_{\geq 0}^n$. The \emph{cubical hull} of $A$ is the smallest lattice box 
\[ B = \left([\ell_1,u_1]\times [\ell_2,u_2] \times \ldots \times [\ell_n,u_n]\right)\cap \Z^n
\]
(with $\ell_i\leq u_i \in \Z_{\geq 0}$) that contains the set $A$. We now show stabilization of tropicalizations of pseudo-moment cones on the unit cube.

\begin{theorem}\label{cubeandcubicalhull}
  Let $Q_m$ be the truncated quadratic module ${\rm QM}_m(x_1,x_2,\ldots,x_n,g_1,g_2,\ldots,g_n)$ with $g_i = 1-x_i$.
  A function $h\colon A \to \R$ belongs to $\pi_A(\trop(Q_m^\vee))$ for all $m\in \N$ with $A\subset m\,\Delta_n$ if and only if it is non-increasing in coordinate directions and 
  it can be extended to a non-increasing midpoint convex function on the cubical hull of $A$. 
\end{theorem}

\begin{proof}
  \Cref{prop:binom} tells us that a function $h\colon A \to \R$ is in $\trop(Q_m^\vee)$ if and only if it is midpoint convex and non-increasing in coordinate directions. Indeed, the localizing matrices $\M{m}{g_i}$ are differences of the Hankel matrix and the localizing matrix of a variable, i.e.~of $\M{m}{1}$ and $\M{m}{x_i}$, which play the roles of $A$ and $B$ in \Cref{prop:binom}. (Formally, we apply \Cref{prop:binom} to the block sum of all of these matrices.) The $2\times 2$ minors of a localizing matrix of a monomial tropicalize to midpoint convexity conditions. The tropicalization of the diagonal entries of $\M{m}{x^Jg_i}$ give the condition that the function $h$ is non-increasing when moving in a coordinate direction.

  So if $h\in \trop(Q_m^\vee)$ for all $m\in \N$, then it is midpoint convex and non-increasing in coordinate directions on $m\,\Delta_{n}$, which contains the cubical hull of $A$ for sufficiently large $m$. This shows the first implication.

  For the other direction, we explicitly describe in \Cref{lem:extensioncubical} an extension from the cubical hull to $\Z_{\geq 0}^n$ that is midpoint-convex and non-increasing. The appropriate restriction of this extension is in $\trop(Q_m^\vee)$ for all $m\in \N$.
\end{proof}

\begin{example}The cubical hull of the Motzkin configuration $A = \{(0,0),(1,1),(2,1),(1,2)\}$
is the square $\{0,1,2\}^2$. The inequalities defining the convex cone of midpoint convex functions $h:\{0,1,2\}^2\to \R$ with $h(\alpha)= y_{\alpha}$ 
are $y_{\alpha} \geq y_{\beta}$ for $\alpha \leq \beta$ (in the partial ordering on $\Z^2$ given by entry-wise comparing the vectors) and $y_{2\alpha}+y_{2\beta}\geq 2 y_{\alpha+\beta}$ for all $\alpha, \beta \in \{0,1\}^2$.
Given values of $h$ on $A$, note that 
$ \min\{y_{12}, y_{21}\}\geq y_{22}$ and $y_{00}+y_{22} \geq 2 y_{11}$. 
Eliminating $y_{22}$ gives inequalities on the values of $h$ on $A$, namely 
$y_{00}+ \min\{y_{12}, y_{21}\}\geq 2 y_{11}$.
We claim that the set of functions $h:A\to \R$ that can be extended to a midpoint convex function on the cube are given by 
\[ y_{11}\geq y_{12}, \ \ y_{11}\geq y_{21}, \ \ y_{00}+y_{12}\geq 2y_{11}, \ \text{ and }\ y_{00}+y_{21}\geq 2y_{11}.
\]
The inequality $y_{00}\geq y_{11}$ is implied by the first and third inequalities. 
Given a function satisfying these inequalities, we can extend it to $\{0,1,2\}^2$
by defining the missing function values to be as large as possible given the 
constraints that $y_{\alpha} \geq y_{\beta}$.
That is, we define $y_{01}$, $y_{02}$, $y_{10}$, and $y_{20}$ 
to be equal to $y_{00}$ as well as $y_{22}$ to be the minimum of $y_{12}$ and $y_{21}$. 
\end{example}

\subsection{Toric Cubes} We now discuss application of the above results to toric cubes, The situtation is more complicated than for 0/1 cube, and we are not able to prove that the stabilization phenomenon holds generally. In Section \ref{sec:stab} we show stabilization in the case that tropicalization of the toric cube is contained in the strictly negative orthant (except for the origin).

Let $S\subseteq [0,1]^n$ be a full-dimensional toric cube defined by $x_i\geq 0$, $1-x_i\geq 0$ and pure binomial inequalities $g_1 \geq 0, \dots, g_r \geq 0$.  It is the image of $[0,1]^d$ with $d\geq n$ under a monomial map $\varphi$. The map $\varphi$ induces a map $P$ between the lattices $\mathbb{Z}^n$ and $\mathbb{Z}^d$. Let $L$ be the image of this map, which is a sublattice of $\Z^d$. We consider $A'=P(A)\subset L$. We would like to apply our previous results to this point configuration in the lattice $L$. The problem is that the cones do not get mapped to each other in general. To set this up, we write $\varphi^*\colon \R[x_1,\ldots,x_n] \to \R[t_1,\ldots,t_d]$ for the induced map on coordinate rings that maps $x_i$ to the $i$-th coordinate monomial in the map $\varphi$. The issue with pulling back our results via the transpose of this map is simply that $\varphi^*\qm(g_1,\ldots,g_r)$ is usually not contained in the quadratic module of the cube $[0,1]^d$.

\begin{example}\label{exm_motztoricsos}
  Let $S$ be the subset of $[0,1]^2$ given by the inequalities $x_2\geq x_1^3$ and $x_1^2\geq x_2$. Then $S$ is the image of $[0,1]^2$ under the map $\varphi\colon [0,1]^2 \to S$, $(t_1,t_2) \mapsto(t_1t_2,t_1^2 t_2^3)$. 
  The corresponding map $P$ of lattices is described by the integer matrix whose columns correspond to the exponent vectors of the monomials in the coordinates of $\varphi$. In this example, we get 
  \[ P = \begin{pmatrix} 1 & 2\\  1 & 3 \end{pmatrix} \]
  This is even a unimodular lattice transformation. In particular, we have $L = \Z^2$ here.

  However, $\varphi^*(x_2 - x_1^3) = t_1^2t_2^3(1-t_1)$ is not in the quadratic module generated by $t_1$, $t_2$, $1-t_1$, and $1-t_2$. To see this, suppose it had a representation 
  \[
    t_1^2t_2^3(1-t_1) = \sigma_0 + \sigma_1 t_1 + \sigma_2 t_2 + \sigma_3 (1-t_1) + \sigma_4 (1-t_2) \]
  with sums of squares of polynomials $\sigma_i \in \R[t_1,t_2]$. By setting $t_1 = 0$ in this identity, we get
  \[ 0 = \sigma_0(0,t_2) + \sigma_2(0,t_2) t_2 + \sigma_3(0,t_2) + \sigma_4(0,t_2)(1-t_2). \]
  Since the right hand side is in a pointed quadratic module (of the interval $[0,1]\subset \R$), every term must be $0$ and therefore, every sum of squares must be identically $0$ which means that it is divisible by $t_1^2$. By plugging back in, we see that every term on the right hand side in the original identity is divisible by $t_1$. Cancelling one $t_1$ from it and repeating the argument, we get a representation
  \[ t_2^3(1-t_1) = \sigma_0' + \sigma_1' t_1 + \sigma_2' t_2 + \sigma_3' (1-t_1) + \sigma_4' (1-t_2). \]
  Symmetrically, we can go through the same process to see that the identity is divisible by $t_2^2$ and end up with the case 
  \[ t_2(1-t_1) = s_0 + s_1 t_1 + s_2 t_2 + s_3 (1-t_1) + s_4(1-t_2) \]
  for sums of squares $s_i \in \R[t_1,t_2]$. To find a contradiction here, we proceed as before by setting $t_2 = 0$ to conclude that the polynomial $(1-t_1)$ would have to be in the quadratic module generated by $t_2$, $t_1t_2$, $(1-t_1)t_2$, and $(1-t_2)t_2$. This is impossible because every generator of this quadratic module vanishes for $t_2 = 0$ which means that $(1-t_1)$ would have to be a sum of squares, contradiction.
\end{example}

We expect that this rather simple example (that could be fixed by considering the preorder of the square $[0,1]^2$ instead of its quadratic module) has generalizations to higher dimensions (where preorders are not saturated anymore) that cannot be fixed in this simple way.

To understand toric cubes, we go back to \Cref{thm:genpseudomom}. Let $k\,\Delta_n\subset \ZZ^n$ denote the set of nonnegative integer vectors with sum of coordinates at most $k$. 

\begin{proposition}\label{prop:ext}
Let $S$ be a toric cube defined by pure binomial inequalities $g_1\geq 0,\dots,g_r\geq 0$. Let $k\in \N$ and $A \subset k\,\Delta_n$. A function $h:A\rightarrow \RR$ belongs to $\trop \pi_A(\qm_k(g_1,\dots,g_r)^\vee)$ if and only if the function $h$ can be extended to a function $\hat{h}:k\,\Delta_n\rightarrow \RR $ satisfying the following three conditions: midpoint convexity, non-increasing in coordinate directions, and non-decreasing in directions of extreme rays of $(\trop S)^\vee$.
\end{proposition}

\begin{proof}
Let $P\in \Z_{\geq 0}^{d\times n}$ be a rank $d$ matrix of whose columns correspond to monomials parameterizing the toric cube $S$. 
By taking logarithms, the parametrization of the toric cube becomes a linear map from $\R^d\to \R^n$ whose matrix is $P^T$. So the tropicalization $\trop S$ is (up to sign) the image of the nonnegative orthant under the map $P^T$. Therefore the dual polyhedral cone $(\trop S)^\vee$ is, by general duality, isomorphic (under $P$) to the intersection of the nonnegative orthant in $\R^d$ with the image of the linear map $P\colon \R^n \to \R^d$.  
The extreme rays of the cone $(\trop S)^\vee$ give binomial inequalities defining $S$ (as a subset of $[0,1]^n$, that is in addition to the inequalities $x_i\geq 0$ and $1-x_i\geq 0$). Applying \Cref{thm:genpseudomom} to the moment matrices in this setup directly results in three types of conditions: midpoint convexity from the $2\times 2$ minors of $\M{}{x^J}$; non-increasing in coordinate directions from the diagonals of $\M{}{x^J(1-x_i)}$; and non-decreasing in directions of extreme rays of $(\trop(S))^\vee$, i.e.\ for anny extreme ray $u-v$ with nonnegative integer vectors $u$ and $v$, we have $\hat{h}(u) \geq \hat{h}(v)$.
\end{proof}

It seems much harder to show stabilization of tropicalized pseudomoment cones for toric cubes than it was for unit cubes (in \Cref{cubeandcubicalhull}). We need to show a bound on $k$ in \Cref{prop:ext}, such that extending to $k\,\Delta_n$ allows us to automatically extend to $(k+i)\,\Delta_{n}$ for all $i \geq 0$. In the next section we do show stabilization in the case that tropicalization of the toric cube is contained in the negative orthant (except for the origin).

\subsection{Stabilization}\label{sec:stab} Let $A$ be a finite subset of $\Z_{\geq 0}^n$ 
and let $C$ be a pointed rational polyhedral cone in $\R^n$. 
Define the cone $K$ to be the intersection of the finitely many translations of $-C$ by the points in $A$: 
\[K = \bigcap_{a\in A} (a-C).\]
Recall that a function $h$ from some subset of $\Z_{\geq 0}^n$ to $\R$ 
is non-decreasing (with respect to $C$) if 
$a-b\in C$ implies that $h(a) \geq h(b)$. 
\begin{proposition} \label{prop:ExtByMax}
Let $B\subseteq \Z_{\geq 0}^n$ be any finite subset containing $A$.  
If a function $h:A\to \R$ can be extended to a midpoint convex, non-decreasing function 
$\widehat{h}:B\to \R$ then it has such an extension with 
$\widehat{h}(b) \geq  \min_{a\in A}h(a)$ for all $b\in B$ and 
$\widehat{h}(b) =\min_{a\in A}h(a)$ for all $b\in K$. 
\end{proposition}
\begin{proof}
Suppose that $h$ has some extension that is midpoint convex and non-decreasing function 
with respect to $C$. The set of  midpoint convex, non-decreasing extensions of $h$ 
is then non-empty. It is also closed under taking point-wise maximum. 
For each $b\in B$, we choose $\widehat{h}_b$ to be an element of this cone as follows. 
If the value of $\widehat{h}(b)$ is bounded from above over all midpoint convex and 
non-decreasing extensions $\widehat{h}$ of $h$, then we choose $\widehat{h}_b$ to be any extension 
achieving this maximum value.  If not, then we choose can choose $\widehat{h}_b$ 
to be an extension such that $\widehat{h}_b(b) > \min_{a\in A}h(a)$. 

Define $\widehat{h}:B\to \R$ to be the point-wise maximum of $\widehat{h}_b$ over all 
$b\in B$. This is a midpoint convex, non-decreasing extension of $h$. 
We claim that  $\widehat{h}(b) \geq \min_{a\in A}h(a)$ for all $b\in B$. 

Let $\lambda$ denote the minimum value of $\widehat{h}(b)$ over all $b\in B$ and let 
 $B_{\lambda}$ denote the set of points achieving this minimum $B_{\lambda}= \{b \in B: \widehat{h}(b) = \lambda\}$. 
The set ${\rm conv}(B_{\lambda}) - C$ is convex. Let $b \in B_{\lambda}$ be 
an element that is an extreme point of this set. 
Some upper bound on the value of $\widehat{h}(b)$ must be tight. If $b\not\in A$, the
potential upper bounds have the form $\widehat{h}(b) \leq \widehat{h}(c)$ where $b-c\in -C$ 
or $\widehat{h}(b) \leq \frac{1}{2}(\widehat{h}(a)+ \widehat{h}(c))$ where $a+c=2b$. 
The first inequality implies that $\widehat{h}(c) = \lambda$ and thus $c\in B_{\lambda}$. 
Since $b = c + (b-c)$, this contradicts the extremality of $b$. 
Similarly, if the inequality $\widehat{h}(b) \leq \frac{1}{2}(\widehat{h}(a)+ \widehat{h}(c))$ is tight, 
then $\widehat{h}(a)=\widehat{h}(c) = \lambda$, showing that both $a$ and $c$ belong to $B_{\lambda}$, and contradicting 
the extremality of $b$.  It follows that $b$ must belong to $A$ and that $\lambda =  \min_{a\in A}h(a)$
This shows that for all $b\in B$, $\widehat{h}(b)\geq  \min_{a\in A}h(a)$. 

By definition, for all points in $b\in K$,  $\widehat{h}(b)\leq  \min_{a\in A}h(a)$, showing equality. 
\end{proof}

In general it may be difficult to determine whether a given midpoint convex, non-decreasing function on a finite set of lattice points can be extended to such a function on all the other lattice points in the nonnegative orthant, which would imply stability of tropicalization of pseudomoment cones for toric cubes.  Here we give a combinatorial condition under which this question can be decided with finite computation.

\begin{proposition}\label{prop:finiteExt}
Suppose that the interior of $-C$ strictly contains the nonnegative orthant, so that the complement of any translate of $K$ in $\Z_{\geq 0}^n$ is finite and define 
\[\widehat{A} =  \left\{2b-a : a, b \in \Z_{\geq 0}^n\backslash K\right\}.   \]
The set $\widehat{A}$ is finite and a function $h:A\to \R$ can be extended to a midpoint convex, non-decreasing (with respect to $C$) function  
$\widehat{h}:\Z_{\geq 0}^n\to \R$ if and only if it can 
be extended to a midpoint convex, non-decreasing function 
$\widehat{h}:\widehat{A}\to \R$. 
\end{proposition}
Note that taking $a = b$ shows that $\widehat{A}$ contains $ \Z_{\geq 0}^n\backslash K$. 
\begin{proof}
Suppose that $h$  can be extended to a midpoint convex, non-decreasing function 
$\widehat{h}:\widehat{A}\to \R$. 
By \Cref{prop:ExtByMax}, we can assume that $\widehat{h}(b) \geq \min_{a\in A}h(a)$
for all $b\in \widehat{A}$ and  $\widehat{h}(b) = \min_{a\in A}h(a)$ for all 
$b\in K\cap \widehat{A}$. 
We then define an extension $\widehat{h}: \Z_{\geq 0}^n\to \R$ by defining $\widehat{h}(b) = \min_{a\in A}h(a)$ for all $b\not\in \widehat{A}$ and 
claim that this extension is both midpoint convex, non-decreasing function. 

To see that it is non-decreasing, suppose that $b,c\in \Z_{\geq 0}^n$ with $c-b\in C$.
If $b\in K$, 
and $\widehat{h}(b) = \min_{a\in A}h(a)$ is the minimum value taken by $\widehat{h}$ on all of $\Z_{\geq 0}^n$. So in particular, 
$\widehat{h}(b) \leq \widehat{h}(c)$. 

If $b\not\in K$, then there is some $a\in A$ for which $b \not\in a-C$. 
This implies that $c$ also does not belong to $a-C$, since $b= c + (b-c)$ and $b-c\in (-C)$. 
If neither $b$ nor $c$ belong $K$, then they both belong to $ \widehat{A}$ 
and $\widehat{h}(b) \leq \widehat{h}(c)$ by assumption. 

To see that $\widehat{h}$ is midpoint convex, suppose that for some $a,b,c\in \Z_{\geq 0}^n$, $a + c = 2b$. 
Note that if  $b\in K$, then the midpoint inequality 
$2\widehat{h}(b) \leq \widehat{h}(a)+\widehat{h}(c)$ is immediately satisfied, 
If $a$ and $c$ belong to $K$, then by convexity, so does $b$.
Otherwise, $b$ and at least one of $\{a,c\}$ belong to $\Z_{\geq 0}^n\backslash K$. Without loss of generality, suppose  $a,b\in \Z_{\geq 0}^n\backslash K$.   Then, by construction, $c = 2b-a \in \widehat{A}$, giving that $2\widehat{h}(b) \leq \widehat{h}(a)+\widehat{h}(c)$. 
\end{proof}

\begin{example}
Consider the exponents $A = \{(0,0), (1,0),(0,1),(1,1)\}$ and the semialgebraic set $S = \{(x_1,x_2)\in \R_{\geq 0}^2 : x^2 < y < x^{1/2}\}$. 
Then $\trop(S)$ is a cone with extreme rays $(-2,-1)$, $(-1,-2)$ and $-C = -\trop(S)^{\vee}$ is a convex cone with extreme rays $(-1,2)$ and $(2,-1)$, which strictly contains the positive orthant. 
Following the construction above, we take $K = \cap_{a\in A}(a-C)$ to be the intersection
of all translations of $-C$ by the points of $A$. In this case, this is just the single translate, $(1,1) - C$. 
There are only finitely many points in $(\Z_{\geq 0}^2 \backslash K)\cup A$, namely $\{(0,0), (1,0), (2,0), (0,1), (1,1), (0,2)\}$.
We then take $\widehat{A} =  \left\{2b-a : a, b \in (\Z_{\geq 0}^2 \backslash K)\cup A\right\}\cap\Z_{\geq 0}^2$, which gives 
\[
\widehat{A} = \{(0,0),(1,0),(2,0),(3,0),(4,0),(0,1),(1,1),(2,1),(0,2),(1,2), (0,3),(0,4)\}.
\]
By \Cref{prop:finiteExt}, a function $h:A\to \R$ can be extended to a function $\widehat{h}:\Z_{\geq 0}^n\to \R$
that is midpoint convex and non-decreasing function with respect to $C = \trop(S)^{\vee}$ 
 if and only if it can 
be extended to a midpoint convex, non-decreasing function 
$\widehat{h}:\widehat{A}\to \R$. 
A computation shows that this set of functions, with is cut out by the six inequalities 
\begin{align*}
&h(1,0)\geq h(1,1), \ \ h(0,0) + h(1,0) \geq 2 h(0,1), \ \  2 h(0,0) + h(1,1) \geq 3 h(1,0),  \\
 &h(0,1)\geq h(1,1), \ \ h(0,0) + h(0,1) \geq 2 h(1,0), \ \ 2 h(0,0) + h(1,1) \geq 3 h(0,1).
\end{align*}
The last inequality, for example, is a convex combination of the inequalities 
$\widehat{h}(1,1) \geq \widehat{h}(0,3)$, $\widehat{h}(0,1) + \widehat{h}(0,3) \geq 2 \widehat{h}(0,2)$ and $\widehat{h}(0,0) + \widehat{h}(0,2) \geq 2 \widehat{h}(0,1)$ on the set of non-decreasing, midpoint convex functions $h: \widehat{A} \to \R$.
The extension to a function $\widehat{h}: \Z_{\geq 0}^2 \to \R$ is given by $\widehat{h}(\alpha) = h(1,1)$ for all $\alpha\not\in \widehat{A}$. 
\end{example}

\begin{example}
For a less symmetric example, consider the Motzkin configuration of moments, $A = \{(0,0),(1,1),(1,2),(2,1)\}$. 
over the semialgebraic set $S = \{(x,y)\in \R_{\geq 0}^2 : y^2 \geq x \geq y^3\}$. 
The cone $-C = -\trop(S)^{\vee}$ has extreme rays spanned by $(1,-2)$ and $(-1,3)$. 
In this case, the intersection $K$ of all translates of $-C$ by the points in $A$ 
equals the translate $(2,1) - C$. 

The set $(\Z_{\geq 0}^2 \backslash K)\cup A$ is finite and given by 
$\{(0,j) : 0\leq j\leq 6\} \cup \{(1,j) : 0\leq j\leq  3\} \cup \{(2,0),(2,1)\}$. To construct the set $\widehat{A}$, 
we take the set of all points completing a midpoint triple with two out of three points coming from this set, giving 
\[
\widehat{A} = \{(0,j) : 0\leq j\leq 12\}\cup  \{(i,j) : i\in \{1,2\}, 0\leq j\leq  6\}\cup  \{(i,j) : i\in \{2,4\}, 0\leq j\leq  2\}.
\]
We project the $33$-dimensional 
convex cone of functions $\widehat{h}:|\widehat{A} |\to \R$ that are midpoint convex and non-decreasingwith respect to $C = \trop(S)^{\vee}$ onto the function values in $A$. This gives a four dimensional cone defined by the inequalities
\begin{align*}
&h(0,0) + 3 h(1,2) \geq 4h(1,1), \ \ 4 h(0,0) + h(1,2) + 5 h(2,1) \geq 10 h(1,1),\\ 
&2h(1,1) + h(2,1) \geq 3 h(1,2), \ \ \text{ and } h(1,2) \geq h(2,1). 
\end{align*}
For example, the second inequality is a convex combination of the inequalities 
$\widehat{h}(1,2) \geq \widehat{h}(0,5)$, $4\widehat{h}(0,0) + \widehat{h}(0,5) \geq 5\widehat{h}(0,1)$ 
and $\widehat{h}(0,1) + \widehat{h}(2,1) \geq 2\widehat{h}(1,1)$ on the function values of 
$\widehat{h}:\widehat{A}\to \R$.  
This four-dimensional cone has a one-dimensional lineality space spanned by $(1,1,1,1)$. 
In $\R^4/\R(1,1,1,1)$ it is a pointed cone over a quadrilateral with extreme rays  
\[(h(0,0), h(1,1), h(1,2),h(2,1)) = (1,0,0,0), \ (20,5,0,0), \ (26,11,6,0), \ (7,3,2,0).\]
Note that $(h(0,0), h(1,1), h(1,2),h(2,1))= (26,11,6,0)$ does not belong to the tropicalization of the moment cone 
$\trop(M_A(S))$ because it is not convex: 
\begin{equation*}
3 h(1,1)=  33 > 32  = h(0,0) + h(1,2) + h(2,1).
\end{equation*}
\end{example}


\section{Further Directions}\label{sec:fd}
To conclude, we highlight some 
open questions. 

One of the main questions left open in the theory developed above is whether 
for an arbitrary set $S\subset \R_{\geq 0}^n$ defined by binomial inequalities and 
a set $A\subset \Z_{\geq}^n$, 
the tropicalizations of the $A$-pseudo-moments of $S$ stabilize as the degree bounds increase. Here is a precise version of this question with the same truncation that we used in \Cref{sec:SOS}.

\begin{question}\label{ques:stabilization}
Let $S\subseteq \R_{\geq 0}^n$ defined by pure binomial inequalities $g_i = x^{a_i} - x^{b_i}$ 
as in \Cref{thm:genpseudomom}.
For any arbitrary finite set $A\subset \Z_{\geq 0}^n$, does there exist $D_A\in \Z_{\geq 0}$ such that 
for all $d\geq D_A$, 
\[
\pi_A( \trop(Q_{D_{A}}^{\vee})) =  \pi_A( \trop(Q_{d}^{\vee}))?
\]
\end{question}
Here $Q_{d}$ denotes  ${\rm QM}_{d}(g_1, \hdots, g_r)$. The question also makes sense, of course, for other ways of bounding the degrees of the sum-of-squares multipliers.

For $S= [0,1]^n$, this is a consequence of \Cref{cubeandcubicalhull}. 
Other sufficient conditions are given \Cref{sec:stab}. 
We note that by \Cref{thm:genpseudomom}, this can be phrased as a question purely 
about polyhedral combinatorics. In the language of \Cref{sec:convexity}, \Cref{ques:stabilization}
can be restated as follows: 

\begin{question}
Let $C$ be a rational polyhedral cone. 
Given an arbitrary finite set $A\subset \Z_{\geq 0}^n$, 
does there exist a finite set $E\subset \Z_{\geq 0}^n$ containing $A$ 
so that for all $E'\supset E$,  
\[
\pi_{A}(\mathcal{M}_{E,C}) = \pi_{A}(\mathcal{M}_{E',C})? 
\]
\end{question}

In the questions above, the set $S$ is fixed and the degree bounds change. 
Another type of stabilization one might consider comes from changing the set $S$. 

\begin{example}\label{ex:motzstab}
Consider the Motzkin configuration $A = \{(0,0),(1,2),(2,1),(1,1)\}$ and the 
sets $S_n = \{(x,y)\in \R_{>0}^2 : y^{n/(n+1)} \leq x \leq y^{(n+1)/n} \}$.  The set $S_n$ is symmetric in 
$x$, $y$ and approaches the line segment between $(0,0)$ and $(1,1)$ as $n\to \infty$. 
The cone $-C_n = -\trop(S_n)^{\vee}$ is spanned by the vectors $(-n, n+1)$ and $(n+1,-n)$. 

For any $d\geq 4$ and $n\geq 2$, we find the following inequalities on functions
$h:\Delta_d \to \R$ in $\mathcal{M}_{\Delta_d, C_n}$:
\[
\tfrac{3}{4} h(0,0) + \tfrac{5}{4}h(2,1) \geq 2 h(1,1), \ \
\tfrac{7}{4} h(2,1) + \tfrac{1}{4} h(0,0) \geq 2 h(1,2), \ \
h(2,1) + h(1,1) \geq 2 h(1,2)
\]
as well as their images under the reflection $h(i,j) \mapsto h(j,i)$. 
The first comes from the midpoint inequality on the triple $\{(0,1),(1,1),(2,1)\}$
and the inequalities  $h(0,1) \leq \frac{1}{4}(3h(0,0)+h(0,4))$ and $h(0,4)\leq h(2,1)$. 
The second and third from from the midpoint inequality on the triple $\{(0,3),(1,2),(2,1)\}$
along with the constraints 
$h(0,3) \leq  \frac{1}{4}(h(0,0)+3h(0,4))\leq \frac{1}{4}(h(0,0)+3h(2,1))$ and 
$h(0,3) \leq h(1,1)$, respectively. 

In $\R^4/\R(1,1,1,1)$, these six inequalities define a convex cone over a hexagon with extreme rays
\[
(h(0,0),h(1,1),h(1,2),h(2,1)) = (1,0,0,0), (8,3,0,0), (8,3,1,0),(8,3,0,1),(8,2,1,0),(8,2,0,1). 
\]
One can check that all of these extend to a mid-point convex function $\widehat{h}:\Z_{\geq 0}^2\to \R$ 
that is non-decreasing with respect to the cone $C_n$. Therefore, these points belong to tropicalization of the pseudo-moment cone for any degree bound $d\geq 4$.
Here we see a different kind of stabilization. 
For every $n\geq 2$, the tropicalization of the $A$-pseudomoment cone is the same. 

Note that the ray spanned by $(8,3,0,0)$ does not satisfy the AM-GM inequality $h(0,0) + h(1,2) + h(2,1) \geq 3 h(1,1)$. 
Since the cones $\pi_A(\mathcal{M}_{\Delta_d,C_n})$ stabilize at $n=2$, this 
inequality cannot be obtained in the limit as $n\to \infty$, even though it does hold 
for the $A$-pseudomoments of the limit set $S_{\infty} = \{(t,t): t\in [0,1]\}$. 
\end{example}
%
It would be interesting to better understand this behavior.


\begin{question}
Give sufficient conditions for a sequence of semialgebraic sets
$(S_n)$ with limit set $S$
such that all convexity inequalities
$\trop\left(M_A(\R_{\geq 0}^n)\right)$
are limits of inequalities in the tropical pseudomoment cone of $S_n$.
\end{question}

Tropicalizations of not necessarily semialgebraic sets with the Hadamard property were considered in \cite{BRST} and \cite{BR21}. In particular it was shown in \cite[Lemma 2.2]{BR21} that the tropicalization of a subset of $\R^{n}_{\geq 0}$ which has the Hadamard property and is closed under addition is a max-closed convex cone, i.e. it is a convex cone which is additionally closed under tropical addition. More generally, we ask whether the tropicalization of a convex cone is necessarily max-closed.

\begin{question}
Let $C\subset \RR^n_{\geq 0}$ be a convex cone. Is it true that $\trop C$ is max-closed?
\end{question}

Many of the tropicalized sets in this paper are both tropically convex 
and convex in the classical sense.  Max-closed convex cones contained in $\RR^n_{\geq 0}$ were studied in \cite{BR21} and \cite{dokuchaev2022cone}. One way of describing a convex body is 
by its extreme points, whose convex hull recovers the set. 
For sets that are convex in both notions, one can ask about the minimal 
generating set from which one can recover the set.  
To this end, we define the double-hull of a set $S\subset \R^n$ to be the tropical convex hull of the convex hull of $S$. 

\begin{question}
What are the double-hull extreme rays (that is the minimal generating set for the double-hull operation) of the convex cones $\mathcal{K}_A$ and $\mathcal{M}_A$, modulo their lineality spaces? 
\end{question}

\bibliographystyle{amsalpha}
\bibliography{tropSOS}

\providecommand{\bysame}{\leavevmode\hbox to3em{\hrulefill}\thinspace}
\providecommand{\MR}{\relax\ifhmode\unskip\space\fi MR }
\providecommand{\MRhref}[2]{%
  \href{http://www.ams.org/mathscinet-getitem?mr=#1}{#2}
}
\providecommand{\href}[2]{#2}
\begin{thebibliography}{DHNdW20}

\bibitem[AGS19]{tropSA}
Xavier Allamigeon, St\'{e}phane Gaubert, and Mateusz Skomra, \emph{The tropical analogue of the {H}elton-{N}ie conjecture is true}, J. Symbolic Comput. \textbf{91} (2019), 129--148.

\bibitem[AGS20]{tropSpec}
\bysame, \emph{Tropical spectrahedra}, Discrete Comput. Geom. \textbf{63} (2020), no.~3, 507--548.

\bibitem[Ale13]{Alessandrini}
Daniele Alessandrini, \emph{Logarithmic limit sets of real semi-algebraic sets}, Adv. Geom. \textbf{13} (2013), no.~1, 155--190.

\bibitem[Bar02]{MR1940576}
Alexander Barvinok, \emph{A course in convexity}, Graduate Studies in Mathematics, vol.~54, American Mathematical Society, Providence, RI, 2002.

\bibitem[BPT13]{SOCAG}
Grigoriy {Blekherman}, Pablo~A. {Parrilo}, and Rekha~R. {Thomas} (eds.), \emph{{Semidefinite optimization and convex algebraic geometry}}, vol.~13, Philadelphia, PA: Society for Industrial and Applied Mathematics (SIAM), 2013.

\bibitem[BR21]{BR21}
Grigoriy Blekherman and Annie Raymond, \emph{A path forward: Tropicalization in extremal combinatorics}, Preprint, arXiv:2108.06377 (2021).

\bibitem[BRST20]{BRST}
Grigoriy Blekherman, Annie Raymond, Mohit Singh, and Rekha~R. Thomas, \emph{Tropicalization of graph profiles}, Preprint, arXiv:2004.05207, to appear in Transactions of the American Mathematical Society (2020).

\bibitem[CF91]{MR1147276}
Ra\'{u}l~E. Curto and Lawrence~A. Fialkow, \emph{Recursiveness, positivity, and truncated moment problems}, Houston J. Math. \textbf{17} (1991), no.~4, 603--635.

\bibitem[CF96]{MR1303090}
\bysame, \emph{Solution of the truncated complex moment problem for flat data}, Mem. Amer. Math. Soc. \textbf{119} (1996), no.~568, x+52.

\bibitem[DDGH13]{copositive2}
Peter J.~C. {Dickinson}, Mirjam {D\"ur}, Luuk {Gijben}, and Roland {Hildebrand}, \emph{{Scaling relationship between the copositive cone and Parrilo's first level approximation}}, {Optim. Lett.} \textbf{7} (2013), no.~8, 1669--1679.

\bibitem[dDS18]{MR3782989}
Philipp~J. di~Dio and Konrad Schm\"{u}dgen, \emph{The multidimensional truncated moment problem: atoms, determinacy, and core variety}, J. Funct. Anal. \textbf{274} (2018), no.~11, 3124--3148.

\bibitem[Dev06]{Develin_secant}
Mike Develin, \emph{Tropical secant varieties of linear spaces}, Discrete Comput. Geom. \textbf{35} (2006), no.~1, 117--129.

\bibitem[DHNdW20]{DHNdW}
Mareike Dressler, Janin Heuer, Helen Naumann, and Timo de~Wolff, \emph{Global optimization via the dual sonc cone and linear programming}, Preprint, arXiv:2002.09368 (2020).

\bibitem[DLRS10]{triangulations}
Jes\'{u}s~A. De~Loera, J\"{o}rg Rambau, and Francisco Santos, \emph{Triangulations}, Algorithms and Computation in Mathematics, vol.~25, Springer-Verlag, Berlin, 2010.

\bibitem[DMP22]{dokuchaev2022cone}
Mikhailo Dokuchaev, Arnaldo Mandel, and Makar Plakhotnyk, \emph{The cone of quasi-semimetrics and exponent matrices of tiled orders}, Discrete Mathematics \textbf{345} (2022), no.~1, 112665.

\bibitem[DS04]{DevelinSturmfels}
Mike Develin and Bernd Sturmfels, \emph{Tropical convexity}, Doc. Math. \textbf{9} (2004), 1--27.

\bibitem[EHS13]{toriccubes}
Alexander {Engstr\"om}, Patricia {Hersh}, and Bernd {Sturmfels}, \emph{{Toric cubes}}, {Rend. Circ. Mat. Palermo (2)} \textbf{62} (2013), no.~1, 67--78.

\bibitem[HLS19]{HLS}
Cvetelina Hill, Sara Lamboglia, and Faye~Pasley Simon, \emph{Tropical convex hulls of polyhedral sets}, Preprint, arXiv:1912.01253 (2019).

\bibitem[IdW16]{IlimanDeWolff}
Sadik Iliman and Timo de~Wolff, \emph{Amoebas, nonnegative polynomials and sums of squares supported on circuits}, Res. Math. Sci. \textbf{3} (2016), Paper No. 9, 35.

\bibitem[JSY22]{JSY}
Philipp {Jell}, Claus {Scheiderer}, and Josephine {Yu}, \emph{{Real tropicalization and analytification of semialgebraic sets}}, {Int. Math. Res. Not.} \textbf{2022} (2022), no.~2, 928--958.

\bibitem[KNT21]{katthan2021unified}
Lukas Katth{\"a}n, Helen Naumann, and Thorsten Theobald, \emph{A unified framework of sage and sonc polynomials and its duality theory}, Mathematics of Computation \textbf{90} (2021), no.~329, 1297--1322.

\bibitem[LS19]{LohoSmith}
Georg Loho and Ben Smith, \emph{Face posets of tropical polyhedra and monomial ideals}, Preprint, arXiv:1909.01236 (2019).

\bibitem[Mar08]{MR2383959}
Murray Marshall, \emph{Positive polynomials and sums of squares}, Mathematical Surveys and Monographs, vol. 146, American Mathematical Society, Providence, RI, 2008.

\bibitem[{Net}09]{zbMATH05598928}
Tim {Netzer}, \emph{{Stability of quadratic modules}}, {Manuscr. Math.} \textbf{129} (2009), no.~2, 251--271.

\bibitem[Put93]{MR1254128}
Mihai Putinar, \emph{Positive polynomials on compact semi-algebraic sets}, Indiana Univ. Math. J. \textbf{42} (1993), no.~3, 969--984.

\bibitem[Rez89]{Reznick89}
Bruce Reznick, \emph{Forms derived from the arithmetic-geometric inequality}, Mathematische Annalen \textbf{283} (1989), no.~3, 431--464.

\bibitem[Sch91]{MR1092173}
Konrad Schm\"{u}dgen, \emph{The {$K$}-moment problem for compact semi-algebraic sets}, Math. Ann. \textbf{289} (1991), no.~2, 203--206.

\bibitem[Sch17]{SchmudgenBook}
\bysame, \emph{The moment problem}, Graduate Texts in Mathematics, vol. 277, Springer, Cham, 2017.

\end{thebibliography}

\end{document}